\documentclass[a4paper, 11pt]{article}
\usepackage[utf8]{inputenc}
\usepackage{geometry}
\usepackage{graphicx}
\usepackage{authblk} 
\usepackage{amsmath}
\usepackage{amssymb}
\usepackage{amsthm}
\usepackage{bm}
\numberwithin{equation}{section}
\usepackage{mathrsfs} 
\usepackage{stmaryrd} 
\usepackage{float}
\usepackage{xcolor}
\usepackage{mathtools}

\usepackage[normalem]{ulem} 

\usepackage{csquotes} 
\usepackage{hyperref}
\hypersetup{colorlinks,citecolor=blue,filecolor=blacklinkcolor=red, urlcolor=magenta,linkcolor=red}

\usepackage{tikz,tikz-cd} 
\usetikzlibrary{positioning}

\newtheorem{defi}{Definition}[section]
\newtheorem{lemm}{Lemma}[section]
\newtheorem{coro}{Corollary}[section]
\newtheorem{theo}{Theorem}[section]
\newtheorem{prop}{Proposition}[section]
\theoremstyle{remark}
\newtheorem{rema}{Remark}[section]

\usepackage[giveninits=true,hyperref,backend=biber,articlein=false,style=ext-alphabetic,maxalphanames=10,maxcitenames=10,maxbibnames=10]{biblatex}
\AtEveryBibitem{%
  \clearfield{issn} 
  \clearfield{Urldate}
  \clearfield{urlmonth}
    \clearfield{month}
  \clearfield{urlyear}
  \clearfield{PrimaryClass}
  \clearfield{eprint}
  \ifentrytype{online}{}{
    \clearfield{url}
  }
}

\newcommand{\dd}{\textnormal{d}}
\DeclareMathOperator{\detd}{\textnormal{\textbf{det}}}

\newcommand{\R}{\mathbb{R}}

\newcommand{\RP}{\mathbb{R}\textnormal{P}}
\newcommand{\T}{\mathcal{T}}
\newcommand{\E}{\mathcal{E}}
\newcommand{\Q}{\mathcal{Q}}
\newcommand\rotpi{\rotatebox[origin=c]{180}{$\Pi$}}
 
\newcommand{\Weyl}[4]{W_{#1#2\phantom{#3}#4}^{\phantom{#1#2}#3}}
\newcommand{\Riem}[4]{R_{#1#2\phantom{#3}#4}^{\phantom{#1#2}#3}}
 
\newcommand{\simeqd}{\mathrel{\rotatebox[origin=c]{-90}{$\simeq$}}}
\newcommand{\veq}{\rotatebox{90}{$\,=$}}

\title{Asymptotic symmetries of projectively compact order one Einstein manifolds} 
\author[1]{Jack Borthwick\thanks{\href{mailto:jack.borthwick@mcgill.ca}{jack.borthwick@mcgill.ca}}}
\author[2]{Yannick Herfray\thanks{\href{mailto:yannick.herfray@univ-tours.fr}{yannick.herfray@univ-tours.fr}}}
\affil[1]{Department of Mathematics and Statistics, McGill University}
\affil[2]{Institut Denis Poisson, Université de Tours}

\date{\today}

\begin{document}
\maketitle
{\let\thefootnote\relax\footnotetext{Keywords: Projective compactification, projective differential geometry, holonomy reductions, curved orbit decompositions }}
\begin{abstract}
We show that the boundary of a projectively compact Einstein manifold of dimension $n$ can be extended by a line bundle naturally constructed from the projective compactification. This extended boundary is such that its automorphisms can be identified with asymptotic symmetries of the compactification. The construction is motivated by the investigation of a new curved orbit decomposition for a $n+1$ dimensional manifold which we prove results in a line bundle over a projectively compact order one Einstein manifolds.
\end{abstract}

\tableofcontents

\section{Introduction}
In the study of objects living on non-compact manifolds it is often enlightening to try to make sense of their asymptotics. Curved orbit decompositions arising from holonomy reductions of (normal) Cartan geometries, as first studied in full generality in~\cite{cap_holonomy_2014}, provide a fascinating framework in which one can seemingly make sense of the asymptotics of geometric structures themselves. In many interesting cases, the orbit decomposition has a large open orbit $\mathcal{O}$ equipped with a certain geometric structure, e.g. an Einstein metric. Its complement can then be viewed as a boundary \enquote{at infinity} and splits into pieces endowed with other geometries that one may interpret as \enquote{limits} of the structure of $\mathcal{O}$.

Defining and studying the asymptotics of pseudo-Riemannian manifolds $(M,g)$ is of particular interest due to their importance in physics. For instance, conformal compactification is key for the formulation of a notion of isolated system in General Relativity~\cite{geroch_asymptotic_1977,penrose_spinors_1984,frauendiener_conformal_2004,valiente_kroon_conformal_2016}. It is also a way to approach the study of the asymptotic behaviour to partial differential equations on non-compact pseudo-Riemannian manifolds, see for instance~\cite{gover_poincare-einstein_2015,curry_introduction_2018,Mason:2004wg,AIF_2016__66_3_1175_0,borthwick2023peeling,mason_nicolas_2009} in the conformal case or~\cite{Borthwick_Proca_Projective} in the projective case, for examples of ideas and work in this direction.

In~\cite{cap_projective_2012,cap_holonomy_2014,Cap:2014aa,Cap:2014ab,cap_projective_2016-1,Flood:2018aa}, a specific type of holonomy reduction of \emph{projective} geometries is studied that gives rise to an Einstein metric on the open orbit, suggesting that projective geometry could play an important role in understanding their asymptotics. 
This is intimately related to the notion of projectively compact metric of order $\alpha$ introduced by Čap--Gover in~\cite{Cap:2014ab}. 
The order of a projectively compact metric is related to its volume asymptotics, and can a priori take any positive real value. For Einstein metrics, the cases $\alpha =1$ (vanishing scalar curvature) and $\alpha=2$ (non-vanishing scalar curvature) are the most relevant. Both cases were thoroughly  investigated  by Čap--Gover in \cite{Cap:2014ab}. In particular, it is shown in this reference that, for projectively compact Einstein metrics of order $1$, this notion coincides with asymptotic flatness at spatial infinity as introduced by Ashtekar and Romano in \cite{ashtekar_spatial_1992}, however with the restriction that only spacetime with vanishing mass aspect can be obtained in this way.

Asymptotic flatness at spatial infinity is by now a classical subject in general relativity, see e.g. \cite{beig_einsteins_1982,compere_relaxing_2011,mohamed_comparison_2021,compere_asymptotic_2023}. A physically important feature of these spacetimes is their asymptotic symmetry group: the -- infinite dimensional -- Spi group of \cite{ashtekar_spatial_1992}. The identification of spatial infinity with projective compactifications of order 1 however presents us with a conundrum: at null infinity, the group of asymptotic symmetry (the BMS group of \cite{bondi_gravitational_1962,sachs_gravitational_1962}) can be identified with the automorphisms of the boundary geometry \cite{geroch_asymptotic_1977,ashtekar_geometry_2015,duval_conformal_2014} and act on the induced Cartan geometry \cite{herfray_asymptotic_2020,herfray_tractor_2022,Herfray:2021qmp} encoding the radiative data. On the other hand the boundary geometry of projective compactification of order 1 of an Einstein metric is again a normal projective geometry with a holonomy reduction corresponding to projective compactification of order $2$ of an Einstein metric~\cite{Cap:2014ab} and thus necessarily has a finite dimensional group of automorphisms. So how exactly does the Spi group of Ashtekar--Romano plays a role in the projective compactification picture?

In this work we wish to resolve this tension by introducing an extended boundary at spatial infinity. In doing so, we follow the footsteps of Ashtekar--Hansen \cite{AhstekarHansen78}. We will show that this idea fits perfectly with, and extends, the notion of projective compactification. The properties of this boundary parallel exactly that of null infinity: its geometry is not rigid and the corresponding group of automorphisms 1) coincides with the Spi group, 2) acts on a space of induced Cartan geometry encoding some asymptotic gravitational data. We also obtain a similar structure at time like infinity.
In order to introduce and motivate this step, we will study a closely related curved orbit decomposition of Cartan geometries \cite{cap_holonomy_2014} generalising the work of Cap--Gover \cite{Cap:2014ab}.
The majority of this article will in fact be dedicated to the discussion of the curved version of a flat model discussed in great detail by Figueroa-O'Farrill and collaborators in \cite{Figueroa-OFarrill:2021sxz}. This homogeneous model generalises the projectively compactified 4D model $\mathbb{R}\text{P}^{4}$ to a 5D model $\mathbb{R}\text{P}^{5}\setminus \{pt\} \simeq \mathbb{R} \times \mathbb{R}\text{P}^{4}$ fibering over the usual projective compactification.
\begin{equation}\label{Introduction: Homogeneous model}
\begin{array}{cccccc}
\mathbb{R}\text{P}^{5}\setminus \{pt\} &\simeq& \left( \mathbb{R} \times M^4\right) &\sqcup&  \left(\textrm{Ti}^4 \sqcup \mathscr{I}^3 \sqcup  \textrm{Spi}^4\right)\\[1em]
\downarrow && \downarrow && \downarrow\\[1em]
\mathbb{R}\text{P}^{4}  &\simeq& M^4 &\sqcup&  \left( H^{3} \sqcup  S^{2} \sqcup  dS^{3}\right).
\end{array}
\end{equation}
This orbit decomposition will be discussed in more detail (and generic dimension) in the core of this article, however we summarise here its essential features: while the action of the Poincaré group has a trivial action along the fibres\footnote{Here $M^4 = \frac{ISO(1,3)}{SO(3,1)}$ is the usual Minkowski spacetime.} $\mathbb{R} \times M^4 \to M^4$ over Minkowski space, it has a non-trivial action over the extended boundary $\textrm{Ti}^4 \sqcup \mathscr{I}^3 \sqcup  \textrm{Spi}^4$:
\begin{equation}\label{Introduction: Boundary Homogeneous space}
\begin{array}{ccccccc}
\textrm{Ti}^4   &\simeq& \frac{\text{ISO}(1,3)}{\mathbb{R}^{3}\rtimes \text{SO}(3)} &\to& H^{3} &\simeq& \frac{\text{SO}(1,3)}{\text{SO}(3)}  \\[1em]

\mathscr{I}^3 &\simeq& \frac{\text{ISO}(1,3)}{\mathbb{R}^{3} \rtimes (\mathbb{R}^* \times \text{ISO}(2)) } &\to& S^{2} &\simeq& \frac{\text{SO}(1,3)}{\mathbb{R}^{2} \rtimes (\mathbb{R}^* \times \text{SO}(2) )} \\[1em]

 \textrm{Spi}^4  &\simeq& \frac{\text{ISO}(1,3)}{ \mathbb{R}^{3} \rtimes \text{SO}(1,2)} &\to& dS^{3} &\simeq& \frac{\text{SO}(1,3)}{\text{SO}(1,2)}.

\end{array}
\end{equation}
This points to the possibility of a non trivial geometrical extension of the boundary geometry of projectively compactified order 1 Einstein metrics. We shall prove in this article that this intuition is correct.

The article is organised as follows. In Section \ref{section: Projective tractors and the homogeneous model} we recall the elements of projective tractor calculus that we will need and discuss in more details the homogeneous model \eqref{Introduction: Homogeneous model}. In Section \ref{Section: curved orbit decompostion} we detail the different curved orbits of the corresponding holonomy reduction, under suitable assumptions this results in a line bundle $M\to \Q$. In Section \ref{section: Geometry of the open orbit O} we prove that the open curved orbit $\mathcal{O}$ is of the form $\mathbb{R}\times \Q_{\mathcal{O}}$ where $\Q_{\mathcal{O}}$ is canonically endowed with an Einstein metric. We then prove, in Section \ref{section: Geometry of Q}, that as a submanifold of $\Q$, this is projectively compact of order $1$ with boundary $\mathcal{H}$. In Section \ref{section: The geometry of Sigma to H} we investigate the induced geometry of the extended boundary $\Sigma \to \mathcal{H}$, first as a submanifold of $M\to \Q$ and then a natural extension of $\mathcal{H} \subset \Q$. Finally, in Section \ref{section: Asymptotic symmetries of projectively compact Einstein manifolds}, we discuss asymptotic symmetries and their action on the extended boundary.

\section*{Acknowledgements}
The first author gratefully acknowledges that this research is supported by NSERC Discovery Grant 105490-2018.

\section{Projective tractors and the homogeneous model}\label{section: Projective tractors and the homogeneous model}

It should be understood that all affine connections on the tangent bundle are torsion-free.

We make extensive use of the Penrose abstract index notation~\cite{penrose_spinors_1984}. In index notation, tractor indices (see below) will be denoted by capital Latin letters, $A,B,C\dots,$. 

The signature of a bilinear form will be written $(q,p)$: \enquote{$q$ `$-$'s and $p$ `$+$'s}.

\subsection{Projective tractor calculus}\label{Projective Tractor Calculus}

We here review the elements of projective tractor calculus that we will need in the rest of the article.

Let $M$ be a manifold of dimension $n+1$ endowed with an equivalence class of projectively equivalent torsion-free connections $[\nabla]$. It is a standard result\footnote{See, for instance,~\cite[Proposition 7.2]{Kobayashi:1995aa}} that two torsion-free affine connections $\hat{\nabla}$ and $\nabla$ are projectively equivalent if and only if there is a 1-form $\Upsilon_a$ such that for any vector field $\xi^b$:
\[\hat{\nabla}_a\xi^b=\nabla_a\xi^b + \Upsilon_a\xi^b +\Upsilon_c\xi^c\delta_a^b. \]
We will use $\hat{\nabla}=\nabla + \Upsilon$ as shorthand for this condition.

We denote by $\mathcal{E}(w)$ the associated vector bundle to the frame bundle $P^1(M)$ determined by the $1$-dimensional representation of $GL_{n+1}(\R)$:
\[ A \mapsto \lvert \det {A}\rvert ^{\frac{w}{n+2}}.\]
Sections of this bundle will be called \emph{projective densities} of weight $w$. Their definition is so that for any $\sigma \in \Gamma(\mathcal{E}(w))$ under the projective change of connection $\hat{\nabla} = \nabla + \Upsilon$, \[\hat{\nabla}\sigma = \nabla \sigma + w\Upsilon \sigma. \]
For any vector bundle $\mathcal{B}$ with base $M$, we will write: $\mathcal{B}(w)=\mathcal{B}\otimes \mathcal{E}(w)$ and we will say that sections of $\mathcal{B}(w)$ are of weight $w$.

Let $G=\textnormal{PGL}(n+2,\R)$ and $H=(\R^{n+1})^*\rtimes GL(n+1,\R)$ the isotropy subgroup of a given line. The class $[\nabla]$ determines a reduction of the second-order frame bundle $P^2(M)$ to a $H$-principle bundle $P$, on which there is a unique \emph{normal} Cartan connection $\omega$~\cite{Cartan:1924aa, Kobayashi:1995aa}; this is a normal projective geometry modeled on projective geometry as described in~\cite{Sharpe:1997aa}.

The \emph{standard (projective) tractor bundle}, $\T$ can be defined as the associated vector bundle to $P$ determined by the restriction to $H$ of the $(n+2)$-dimensional representation of $G$ given by:  $A\mapsto \lvert\det A\rvert^{-\frac{1}{n+2}}A$. The Cartan connection $\omega$ induces a linear connection $\nabla^{\mathcal{T}}$ on $\mathcal{T}$.

The bundle $\mathcal{T}$ fits into a canonical short exact sequence of vector bundles:
\begin{center}
\begin{tikzcd} 0 \arrow[r] & \mathcal{E}(-1) \arrow[r,"X"] & \arrow[l,dotted,bend left = 45, "Y"] \mathcal{T} \arrow[r,"Z"] & TM(-1) \arrow[r] \arrow[l,dotted,bend left=45,"W"] &0. \end{tikzcd}
\end{center}
A choice of connection in the projective class $[\nabla]$ provides a non-canonical isomorphism $\mathcal{T}\simeq \mathcal{E}(-1)\oplus TM(-1)$, described above by the non-canonical maps $Y$ and $W$. When thinking of these maps as sections of $\mathcal{T}^*\otimes\mathcal{E}(-1)$ and $T^*M(1)\otimes\mathcal{T}$ respectively, they transform under change of connection $\hat{\nabla}=\nabla + \Upsilon$ according to:
\begin{equation}\hat{Y}_A=Y_A - \Upsilon_aZ^a_A, \quad \widehat{W}^A_a=W^A_a + \Upsilon_a X^A. \end{equation}
Fixing a choice of connection $\nabla$ in the projective class, sections of the standard tractor bundle will be written:
\[ T^A=\xi^aW_a^A + \rho X^A,\]
similar expressions will be used with sections of tensor powers of the tractor bundle.

The dual tractor bundle $\T^*$ is usually identified with the first jet prolongation $J^1\E(1)$ of the weight $1$ projective density bundle~\cite{Bailey:1994aa}. This description has the advantage of being more direct than the principal bundle approach outlined above, and is independent of the projective structure. In addition, it allows for straightforward comparisons of tractors on different manifolds; in particular, tractors in the \enquote{bulk} and on distinguished hypersurfaces.

We also recall from~\cite{Bailey:1994aa} the standard decomposition of the curvature tensor $R_{ab\phantom{c}d}^{\phantom{ab}c}$:
\[ R_{ab\phantom{c}d}^{\phantom{ab}c}=W_{ab\phantom{c}d}^{\phantom{ab}c} + 2\delta^c_{[a}P_{b]d} + \beta_{ab}\delta^c_d,\]
where the Weyl tensor $W$ is trace-free and $P_{ab}$ is the projective Schouten tensor. It is easily verified that:
\[nP_{ab}= R_{ab}+\beta_{ab},\quad \beta_{ab}=-\frac{2}{n+2}R_{[ab]}=-P_{[ab]},\]
where $R_{ab}=R_{da\phantom{d}b}^{\phantom{da}d}$ is the Ricci tensor. The tensor $\beta_{ab}$ can be thought of as the curvature tensor of the density bundle $\mathcal{E}(1)$.

A connection $\nabla$ is said to be \emph{special} if it preserves a nowhere vanishing density $\sigma$. In this case $\beta_{ab}=0$ (in particular all density bundles are flat) and we say that $\nabla$ is the \emph{scale} determined by $\sigma$. Even when it does not preserve a nowhere vanishing density, we will refer to a choice of connection in the class as a choice of scale.

Having fixed a projective connection $\nabla$ in the class, and split the short-exact sequence, the action of the normal Cartan connection can be summarised conveniently by the relations:
\begin{equation}\label{TractorConnection} \nabla_a X^A= W^A_a, \quad \nabla_a W^B_b=-P_{ab}X^B,\quad \nabla_a Y_A=P_{ab}Z^b_A, \quad \nabla_a Z^b_B=-\delta_{a}^bY_B. \end{equation}

The tractor curvature, defined by, the identity $\Omega_{ab\phantom{C}D}^{\phantom{ab}C}T^D= 2 \nabla_{[a}\nabla_{b]}T^C$, is known to be given in an arbitrary scale by:
\begin{equation}\label{TractorCurvature} \Omega_{ab\phantom{C}D}^{\phantom{ab}C}=\Weyl{a}{b}{c}{d}W^C_cZ^d_D -Y_{abd}X_CZ^d_D \end{equation}
Where: $Y_{abd}=2\nabla_{[a}P_{b]d}$ is the projective Cotton tensor.

A projective class $[\nabla]$ is said to be \emph{metric}, if it admits a solution $\zeta^{ab} \in \Gamma(S^2(T^*M)(-2))$ to the Mike\v{s}-Sinjukov metrisability equation~\cite{Mikes1996,Sinjukov:1979aa}:
\[ \textnormal{trace-free}\left( \nabla_c\zeta^{ab} \right)=0, \]
or equivalently $\nabla_c \zeta^{ab} = 2 \delta^{(a}_c \lambda^{b)}$ for some $\lambda^a$.
Solutions of this equation have been shown~\cite{Eastwood:2008aa,Flood:2018aa} to be in one-to-one correspondence with tractors $H^{AB}$ that satisfy:
\begin{equation} \label{TractorMetrisability}\nabla_c H^{AB} + \frac{2}{n+1}X^{(A}W_{cE\phantom{B}F}^{\phantom{cE}B)}H^{EF}=0, \end{equation}
where $W_{cE\phantom{B}F}^{\phantom{cE}B}=Z^c_C\Omega_{ce\phantom{B}F}^{\phantom{ce}B}$ and $\Omega_{ce\phantom{B}F}^{\phantom{ce}B}$ is the tractor curvature.
In this case $H^{AB}$ is of the form:
\[H^{AB}= \zeta^{ab}W^A_aW^B_b + \lambda^aX^{(A}W^{B)}_b+ \tau X^AX^B.\]
with 
\begin{align}\label{H decomposition}
\lambda^a &= - \frac{2}{n+2}\nabla_c\zeta^{ca},& \tau&=  \frac{\nabla_a\nabla_b \zeta^{ab}}{(n+1)(n+2)}+\frac{1}{n+1}P_{ab}\zeta^{ab}.
\end{align}

If $\nabla_c H^{AB}=0$, then it can be shown that Eq.~\eqref{TractorMetrisability} is automatically satisfied; these solutions are said to be \emph{normal}. This is equivalent to
\begin{align}\label{DH=0}
\nabla_c \zeta^{ab} - 2 \delta^{(a}_c \lambda^{b)} &=0, &
\nabla_c \lambda^a + \delta^a_c \tau - P_{cb}\zeta^{ba} &=0,
\end{align}
(and imply \eqref{H decomposition}). Normal solutions are known to correspond to Einstein metrics~\cite{Cap:2014aa,Armstrong2008}. If one introduces\footnote{$\epsilon^2$ is the section that realises the identification $(\Lambda^{n+1}TM)^{\otimes 2} \simeq \mathcal{E}(2(n+2))$; if $M$ is orientable, then it can be thought of as the square of the volume form.}
\begin{equation*}
\detd : \begin{array}{ccc}
\mathcal{E}^{ab}(w) & \longrightarrow &\mathcal{E}((n+1)(w+2) + 2) \\
h^{ab} &\longmapsto& \epsilon^2_{a_1 ... a_{n+1}b_1... b_{n+1}} h^{a_1b_1}... h^{a_{n+1}b_{n+1}}
\end{array}
\end{equation*}
and\footnote{We recall from \cite{Flood:2018aa} that the projective tractor volume form is defined by \begin{equation*} \epsilon^2_{A_0 A_1 ... A_{n+1}B_1 B_1... B_{n+1}} :=  \epsilon^2_{a_1 ... a_{n+1}b_1... b_{n+1}} Y_{[A_0} Z_{A_1}^{a_1} ...  Z_{A_{n+1}]}^{a_{n+1}} Y_{[B_0} Z_{B_1}^{b_1} ...  Z_{B_{n+1}]}^{b_{n+1}}
	\end{equation*}}
\begin{equation*}
	\det : \begin{array}{ccl}
		\mathcal{E}^{AB} & \longrightarrow & C^{\infty}(M) \\
		H^{AB} &\longmapsto& \epsilon^2_{A_0 A_1 ... A_{n+1}B_1 B_1... B_{n+1}} H^{A_0B_0}... H^{A_{n+1}B_{n+1}}
	\end{array}
\end{equation*}
then $\detd(\zeta)\zeta^{ab} \in \mathcal{E}^{ab}$ is an Einstein metric with scalar curvature $R= n(n+1) \det(H)$. 

It will also be useful to keep in mind the following algebraic result (which we take from \cite{Flood:2018aa}): \begin{lemm}If $H^{AB}$ is invertible of inverse $\Phi_{AB}$ then $X^A X^B \Phi_{AB} =\det(H) \detd(\zeta)$ and if $\det(H)=0$ then the kernel is generated by $D_{A}\tau$ with $\tau^2 = \frac{\detd(\zeta)}{n+1}$.
\end{lemm}
\subsection{The Model}\label{model}

We now review from Figueroa-O'Farrill and collaborators \cite{Figueroa-OFarrill:2021sxz} the homogeneous model on which is based the holonomy reduction of Cartan geometry which is considered in this article. We only give a brief outline of the results that are of interest for this article and refer the reader to \cite{Figueroa-OFarrill:2021sxz} for an extensive discussion and proofs.

We will write elements of the $n+1$-dimensional projective space $\mathbb{R}\text{P}^{n+1} \simeq \frac{\text{SL}(n+2)}{\mathbb{R}^{n+1} \rtimes \text{GL}(n+1)}$ in terms of homogeneous coordinates $[X]$ with $X\in \mathbb{R}^{n+2}$.
 Introducing a metric $\Phi$ of signature $(2,n)$ on $\mathbb{R}^{n+2}$ together with a null element $I \in \mathbb{R}^{n+2}$, $\Phi(I,I)=0$ picks a Poincaré subgroup
\begin{equation}
\text{ISO}(1,n-1) \subset \text{SL}(n+2).
\end{equation}
stabilising both $\Phi$ and $I$.

Since by definition the action of the Poincaré group $\text{ISO}(1,n-1)$ stabilises $I$ this defines a point-like orbit $[I]$ in  $\mathbb{R}\text{P}^{n+1}$. Introducing \[\mathcal{M}_{\lambda, \sigma} = \{ X\in \mathbb{R}^{n+2} \;s.t.\; \Phi(X,X)=\lambda,  \Phi(X,I)=\sigma, X\not \propto I\}\] the orbit decomposition of  $\mathbb{R}\text{P}^{n+1}$ under $\text{ISO}(1,n-1)$ is given by
\begin{equation}\label{ModelOrbitDecomposition}
\mathbb{R}\text{P}^{n+1}  \simeq \left(\sqcup_{\lambda \in \mathbb{R}} \mathcal{M}_{\lambda, \sigma=1}\right) \sqcup  \left(\mathcal{M}_{-1, 0} \sqcup  \mathcal{M}_{0, 0} \sqcup  \mathcal{M}_{1, 0}\right) \sqcup  [I].
\end{equation}
Each of the orbits in the first continuous set is isomorphic to $n$-dimensional Minkowski space:
\begin{align}
\forall \lambda &\in \mathbb{R}, \quad \mathcal{M}_{\lambda, \sigma=1}  \simeq \frac{\text{ISO}(1,n-1)}{\text{SO}(1,n-1)}.
\end{align}
The next three orbits are ``extended boundaries'', respectively denoted as $\text{Ti}^{n} = \mathcal{M}_{-1, 0}$, $\mathscr{I}^{n-1} = \mathcal{M}_{0, 0}$ and $\text{Spi}^{n} = \mathcal{M}_{1, 0}$ in \cite{Figueroa-OFarrill:2021sxz}. These are naturally real line bundles over hyperbolic space, the conformal sphere and de Sitter space respectively:
\begin{equation}\label{Boundary Homogeneous space}
\begin{array}{rcccccc}
\textrm{Ti}^n  &\simeq& \frac{\text{ISO}(1,n-1)}{\mathbb{R}^{n-1}\rtimes \text{SO}(n-1)} &\to& H^{n-1} &\simeq& \frac{\text{SO}(1,n-1)}{\text{SO}(n-1)}  \\[1em]

\mathscr{I}^{n-1} &\simeq& \frac{\text{ISO}(1,n-1)}{\mathbb{R}^{n-1} \rtimes (\mathbb{R}^* \times \text{ISO}(n-2)) } &\to& S^{n-2} &\simeq& \frac{\text{SO}(1,n-1)}{\mathbb{R}^{n-2} \rtimes (\mathbb{R}^* \times \text{SO}(n-2) )} \\[1em]

\textrm{Spi}^n  &\simeq& \frac{\text{ISO}(1,n-1)}{ \mathbb{R}^{n-1} \rtimes \text{SO}(1,n-2)} &\to& dS^{n-1} &\simeq& \frac{\text{SO}(1,n-1)}{\text{SO}(1,n-2)}.

\end{array}
\end{equation}

We now briefly recall how this orbit decomposition relates both to conformal and projective compactifications of Minkowski space.

The action of the Poincaré group on the Klein quadric $N = \{ [X]\in \mathbb{R}\text{P}^{n+1} \;s.t.\; \Phi(X,X)=0\}$ gives an orbit decomposition
\begin{equation}
N \simeq \mathcal{M}_{\lambda=0, \sigma=1} \sqcup  \mathcal{M}_{0, 0} \sqcup [I].
\end{equation}
This realises the conformal compactification of Minkowski space and was taken as the model for the holonomy reduction studied in \cite{herfray_asymptotic_2020,herfray_tractor_2022,Herfray:2021qmp}.

On the other hand, as depicted in Figure \ref{Figure: homogeneous space}, $\mathbb{R}\text{P}^{n+1}\setminus [I]$ is naturally a fibre bundle over $\mathbb{R}\text{P}^{n}$, where the projection is obtained by quotienting by $I$. This results in a fibration of the orbit decomposition (also depicted, in the dimensionally reduced $n=1$ case, in Figure \ref{Figure: 2D picture of the model}):
\begin{equation}\label{Homogeneous model: diagram decomposition}
\begin{array}{cccccc}
\mathbb{R}\text{P}^{n+1}\setminus [I]  &\simeq& \left(\sqcup_{\lambda \in \mathbb{R}} \mathcal{M}_{\lambda, \sigma=1}\right) &\sqcup&  \left(\mathcal{M}_{-1, 0} \sqcup \mathcal{M}_{0, 0} \sqcup  \mathcal{M}_{1, 0}\right)\\[1em]
\downarrow && \downarrow && \downarrow\\[1em]
\mathbb{R}\text{P}^{n}  &\simeq& M^n &\sqcup&  \left( H^{n-1} \sqcup  S^{n-2} \sqcup  dS^{n-1}\right)
\end{array}
\end{equation}
This quotient therefore realises the projective compactification of Minkowski space which was the starting point for the holonomy reduction studied in \cite{Cap:2014aa,Cap:2014ab,Flood:2018aa}.

We therefore see that this extended model in a sense bridges between the conformal and projective compactification of Minkowski space. The investigation of the corresponding interplay between conformal and projective aspects in the associated curved orbit decomposition is the subject of this article.

\begin{figure}[h!]
\centering
\begin{tikzcd}[column sep={3cm,between origins}, row sep={2.5cm, between origins},arrows={very thick,-latex}]
& \R^{n+2}\setminus\{\R I\}\arrow[dl,swap,color=gray,"/\R"] \arrow[dr,"/I"]&\\
\underset{\substack{\textrm{``Lines in $\R^{n+2}$}\\ \textrm{differing from $[I]$ ''}}}{\RP^{n+1}\setminus[I]} \arrow[dr,swap,"/I"] & & \R^{n+1}\setminus\{0\}\arrow[dl,color=gray,"/\R"]
\\ & \underset{\substack{\textrm{``Planes in $\R^{n+2}$}\\ \textrm{containing $I$''}}}{\RP^n} & 
\end{tikzcd}
\caption{Fibration of $\RP^{n+1}\setminus[I]$ over $\RP^{n}$ and their corresponding ambient spaces.}\label{Figure: homogeneous space}
\end{figure}
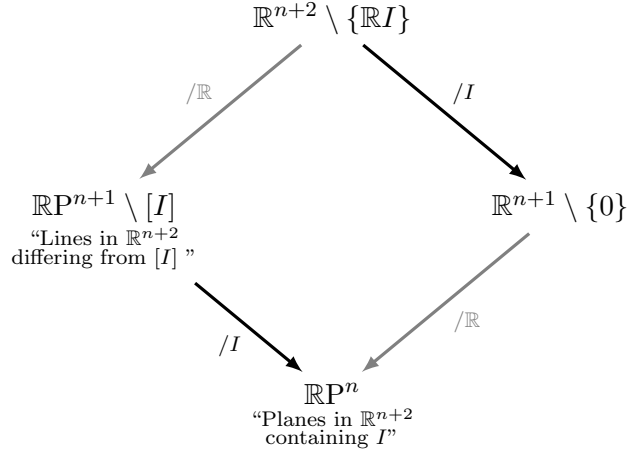

\begin{figure}[h!]
	\centering
	\includegraphics[width=1\textwidth]{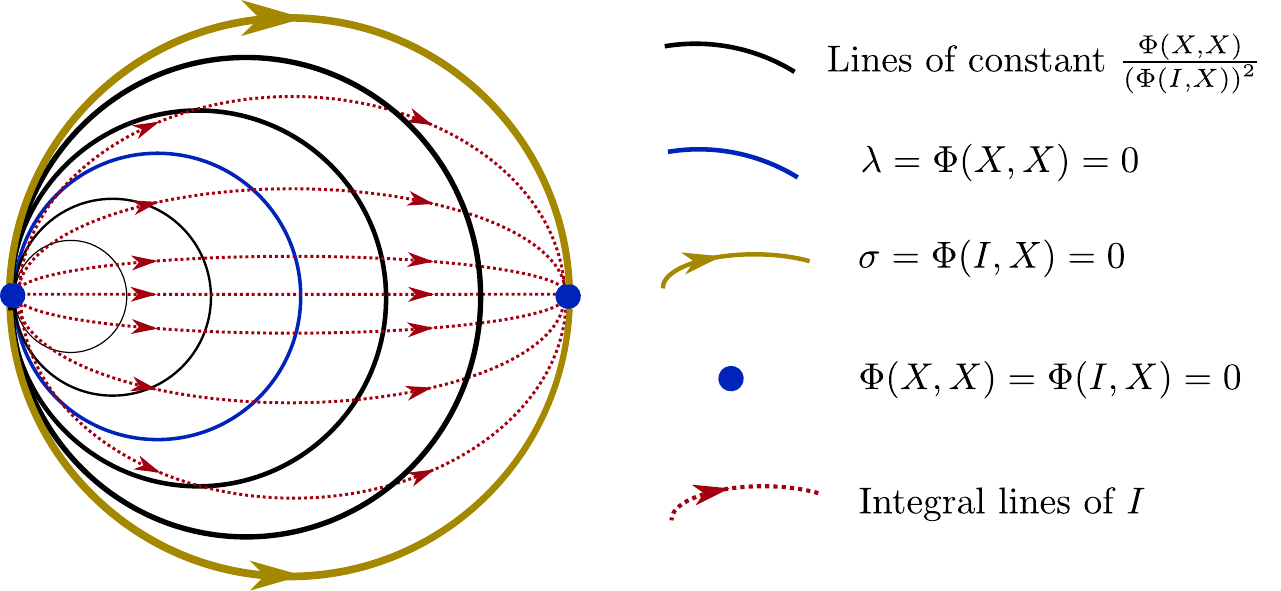}
	\caption{Dimensionally reduced depiction of $\RP^{n+1}\setminus[I]$. The oriented lines indicate the fibres of the projection  $\RP^{n+1}\setminus[I]\to\RP^{n}$; apart from those along the extended boundary at infinity -- depicted here by a golden line -- the fibres cross successive surfaces of constant growing $\frac{\Phi(X,X)}{(\Phi(I,X)^2}$ (ranging from $-\infty$ to $+\infty$).}\label{Figure: 2D picture of the model}
\end{figure}

\newpage
\section{Curved orbit decomposition}\label{Section: curved orbit decompostion}
The object of this section is to study the curved picture of the previous homogenous model. Throughout, $\widetilde{M}$ is a compact manifold of dimension $n+1$ equipped with a projective structure $[\nabla]$. We assume that it is additionally endowed with a parallel  \emph{non-degenerate} tractor $H^{AB}$ of signature\footnote{For definiteness, we work in this article with the physically interesting case of Lorentz signature. However, most of our considerations are independent of this choice.} $(2,n)$ corresponding to a normal of the metrisability equation, and a nowhere vanishing parallel null-tractor $I^A$:
\begin{align}\label{Holonomy reduction conditions DH=0, DI=0}
\nabla_a H^{AB} &=0, & \nabla_a I^A &=0,
\end{align}
 with $\det(H^{AB})\neq0$ and  $I^AI^B\Phi_{AB}=0$, where $\Phi_{AB}$ is the pointwise inverse of $H^{AB}$.
 
From Eq.~\eqref{TractorConnection}, it is easily seen that in an arbitrary scale $\nabla$ the components of
\[ I^A= n^aW_a^A + \rho X^A,\]
satisfy the equations:
\begin{equation}\label{IParallel} \nabla_cI^A =\left(\nabla_c n^a + \delta_c^a \rho \right)W_a^A + (\nabla_c \rho - P_{ca}n^a)X^A=0 \Leftrightarrow \begin{cases} \nabla_c n^a= -\delta_c^a \rho \\ \nabla_c\rho = P_{ca}n^a \end{cases}. \end{equation}
We emphasise that $n^a \in \Gamma(T\widetilde{M}(-1))$ is a scale-independent quantity whilst $\rho$ is not. On top of this vector field, we also have at our disposal two distinguished densities (of weight 1 and 2 respectively) that will be central to our development:
\begin{align}
\sigma&=\Phi_{AB}I^AX^B,\label{DefSigma}\\
\lambda&=\Phi_{AB}X^AX^B.\label{DefLambda}
\end{align}

These will allow us to state our main result concerning the curved orbits. 

\subsection{Resulting curved orbits}

\begin{theo}\mbox{}\label{Thrm: Curved orbits decomposition} Let $\widetilde{M}$ be a $(n+1)$-dimensional closed manifold equipped with a projective structure $[\nabla]$, a non-degenerate parallel tractor $H^{AB}$ of signature $(2,n)$ and a nowhere vanishing parallel tractor $I^A$, such that $\Phi_{AB}I^AI^B=0$ where $\Phi_{AB}$ is the pointwise inverse of $H^{AB}$. Define, $n^a$, $\sigma$ and $\lambda$ by Eqs.~\eqref{IParallel}--\eqref{DefLambda}.
\begin{enumerate}
\item The nowhere-vanishing null parallel tractor $I^A$ determines a first decomposition of $\widetilde{M}$ into two parts:
\begin{equation*} \widetilde{M} = M\sqcup \mathcal{Z}(n),
\end{equation*}
where $\mathcal{Z}(n) := \{x\in \widetilde{M}, n^a=0 \}$ is composed of finitely many isolated points and $M$ is the complement. 
\item The additional data of the parallel tractor $H^{AB}$ then provides a further decomposition
\begin{equation}
M= \mathcal{O} \sqcup \Sigma
\end{equation}
 Here $\mathcal{O}$ is the open submanifold such that $\sigma \neq 0$ and (if non-empty) $\Sigma := \mathcal{Z}(\sigma)\setminus \mathcal{Z}(n)$ is a smooth embedded hypersurface. 
 \item Finally, $\Sigma$ itself decomposes into further smooth submanifolds: 
 \begin{equation*}
\Sigma= \Sigma^- \sqcup \Sigma^0 \sqcup \Sigma^+
 \end{equation*}
   with $\Sigma^- :=\mathcal{Z}(\sigma)\cap\{\lambda <0\}$, $\Sigma^+:=\mathcal{Z}(\sigma)\cap\{\lambda > 0\}$  and $\Sigma^0:=\mathcal{Z}(\sigma)\cap\mathcal{Z}(\lambda)\setminus\mathcal{Z}(n).$
  \end{enumerate}
 Overall, in summary, we have:
  \begin{equation}\label{Decomposition of the manifold}
  M=\mathcal{O} \sqcup \Sigma^- \sqcup \Sigma^0 \sqcup \Sigma^+ \sqcup \mathcal{Z}(n).
  \end{equation}
\end{theo}

It should be noted that $\widetilde{M}$ also admits another curved orbit decomposition relative to the data $(\lambda, H^{AB})$:
\begin{equation}\label{Decomposition of O}
\widetilde{M} = U^-\sqcup \Lambda\sqcup U^+,
\end{equation}
 where $U^-$, $\Lambda$ and  $U^-$ are the sets $\{\lambda<0\}$, $\{\lambda=0\}$ and $\{\lambda>0\}$ respectively.  In particular, it follows from \cite{Cap:2014ab,Cap:2014aa,Flood:2018aa} that $\Lambda$ is the projective boundary of $U^{\pm}$. This would additionally lead to a further decomposition of $\mathcal{O}$ as $\mathcal{O}^+\sqcup \mathcal{O}^0 \sqcup \mathcal{O}^- $, however, this will not be our central focus because the decomposition \eqref{Decomposition of O} is \emph{not} preserved by the flow\footnote{in the sense described in Section~\ref{Proof3point2}} of the densified vector field $n^a$. Rather we have:

\begin{theo}\mbox{}\label{Thrm: Curved orbits decomposition: quotient}
The flow of $n^a$ preserves the decomposition \eqref{Decomposition of the manifold}.  Furthermore, the quotient $\Q_\mathcal{O}$ of $\mathcal{O}$ by this flow is a $n$-dimensional manifold.

If we  assume that the quotient $\Q$ of $M=\widetilde{M}\setminus \mathcal{Z}(n)$ is a manifold we obtain the decomposition
\begin{center}
\begin{tikzcd}[column sep=2mm]
M\arrow[d] &=&\mathcal{O}\arrow[d] &\sqcup& \overbrace{ \left(\Sigma^- \sqcup \Sigma^0 \sqcup \Sigma^+\right)}^\Sigma\arrow[d]\\
\Q  &\simeq& \Q_\mathcal{O} &\sqcup&  \underbrace{\left(  \mathcal{H}^- \sqcup  \mathcal{H}^0 \sqcup   \mathcal{H}^+\right)}_{\mathcal{H}}
\end{tikzcd}
\end{center}
generalising \eqref{Homogeneous model: diagram decomposition} in our curved setting.
\end{theo} 

Our present goal is to identify the geometric structures of each part and the relationships between them when they are all (save possibly $\mathcal{Z}(n)$) non-empty.

\subsection{Derivation of Theorem \ref{Thrm: Curved orbits decomposition}}
We recall that the Thomas $D$-operator is defined by the following equation in an arbitrary scale:
\begin{equation}\label{ThomasD} D_A T^{\mathscr{B}} = \nabla_a T^{\mathscr{B}}Z^a_A +w T^{\mathscr{B}}Y_A.\end{equation}
Where $T^{\mathscr{B}}$ is a section of an arbitrary bundle with tractor indices, of weight $w$; note that $D_A$ reduces weight by $1$. The following lemma will be useful: \begin{lemm}\label{Lemma: Relations ISigma, XLambda, etc} Let $H^{AB}, I^A$ be the parallel tractors (see Eq.~\eqref{Holonomy reduction conditions DH=0, DI=0}) introduced at the beginning of Section~\ref{Section: curved orbit decompostion} and $\sigma, \lambda$ the densities defined by Eqs.\eqref{DefSigma} and \eqref{DefLambda}, then:
\begin{align}\label{RelationISigma} I^A= H^{AB}D_B\sigma,\\ \label{RelationXLambda}
\Phi_{AB}X^B=\frac{1}{2}D_A\lambda, \\ \label{InverseTractorMetric} \Phi_{AB}=\frac{1}{2}D_AD_B \lambda. \end{align}

\end{lemm}
\begin{proof}
In any scale, the identity tractor satisfies: \[ \delta^A_B=W^A_aZ_B^b\delta_b^a + X^AY_B=D_BX^A.\] 
Since the Thomas $D$ operator satisfies the Leibniz rule, we have:
\[D_A\sigma = D_A \Phi_{BC}I^BX^C + \Phi_{BC}D_AI^BX^C+\Phi_{BC}I^BD_AX^C\]
However, $\Phi$ and $I$ are weightless parallel tractors, so, by definition of the Thomas $D$ operator, $D_A\Phi_{BC}=0$ and $D_A I^B=0$.
Thus:
\[D_A\sigma = \Phi_{BA}I^B. \]
Multiplying by $H^{AB}$ leads to the first equation.

In a similar fashion: 
\[D_A\lambda= 2\Phi_{BC}(D_AX^B)X^C=2\Phi_{AC}X^C \]
Applying $D_B$ once more to this equation gives the last identity.
\end{proof}

We use this to prove the first elementary decomposition of the manifold $M$.
\begin{prop}
For any connection $\nabla$ in the projective class $\Sigma \cap \mathcal{Z}(\nabla_c\sigma) = \emptyset$. Consequently, when non-empty, $\Sigma$ is a smooth hypersurface of $M$.
\end{prop}
\begin{proof}
From~\eqref{RelationISigma}, we see that on $\mathcal{Z}(\sigma)\cap\mathcal{Z}(\nabla_c\sigma)$ we must have:
\[ 0= D_A\sigma= \Phi_{AB}I^B.\]
Since by assumption, $\Phi_{AB}$ is non-degenerate this implies $I^B=0$ at such points. The result follows as we assume $I^B$ non-vanishing.
\end{proof}

The hypersurface $\Sigma$ will be interpreted as points at infinity of the region $\{\sigma \neq 0\}$; this is justified by:
\begin{prop}
Let $\mathcal{O}_\pm=\{\pm \sigma > 0\}$, then the scale $\nabla^\sigma$ determined by $\sigma$ on $\mathcal{O}_\pm$ is projectively compact of order one and $\mathcal{O}_\pm \cup \Sigma$ is a projective compactification of $\mathcal{O}_\pm$. \end{prop}
\begin{proof}
This is a direct consequence of~\cite[Proposition 2.4]{Flood:2018aa}.
\end{proof}
We recall from~\cite[Proposition 2.3]{Cap:2014ab} that in this case the points are indeed \enquote{at infinity} with respect to the (parametrised) geodesics of $\nabla^\sigma$, they cannot be reached for a finite value of the affine parameter.

Whilst the decomposition of $\mathcal{O}$  associated with $H^{AB}$ and the density $\lambda$ studied in~\cite{Flood:2018aa} is not of direct interest to us, the set $\mathcal{Z}(\sigma) \cap \mathcal{Z}(\lambda)$ is an important region. Ideally, we would like it to be a codimension $2$ submanifold of $\widetilde{M}$, the following proposition studies some possible singular points.
\begin{prop}
Recall that $I^A$ defines a weight $-1$ vector field $n^a=Z^a_AI^A$.
The set of points $\mathcal{Z}(n)$, or equivalently points at which $I\propto X$, is composed of isolated points that lie in $\mathcal{Z}(\lambda)\cap\mathcal{Z}(\sigma)$.
Since $\widetilde{M}$ is compact $\mathcal{Z}(n)$ is finite.
\end{prop}
\begin{proof}
We introduce an arbitrary scale $\nabla$, and recall that
	$I^A=\rho X^A + n^aZ_a^A$. Thus $n$ vanishes exactly at points where $I\propto X$. 
	
	Since $I$ is null, and by definition of $\sigma$, when $I\propto X$ it follows that we have both $\lambda=\Phi_{AB}X^AX^B \propto \Phi_{AB}I^AI^B=0$ and $\sigma=\Phi_{AB}X^AI^B \propto \Phi_{AB}X^AX^B=0.$ So $\mathcal{Z}(n)\subset \mathcal{Z}(\lambda)\cap\mathcal{Z}(\sigma)$. Let us prove that these are isolated points. 

Using Eq.~\eqref{TractorConnection} we see that:
\[\nabla_b n^a= \nabla_b(I^AZ_A^a)=-\delta_b^aY_AI^A=-\delta_b^a \rho.\]
Since $I$ is nowhere vanishing, $\rho$ does not vanish on $\mathcal{Z}(n)$ and for any vector field $v^b$ that does not vanish on $\mathcal{Z}(n)$, we must have:
$v^b\nabla_b n^a=\rho v^b \neq 0$ therefore $\mathcal{Z}(n)$ consists of isolated points.
\end{proof}

The following justifies the introduction of $\mathcal{Z}(n)$:
\begin{prop}
Let $\Sigma^0$ be the relative complement of $\mathcal{Z}(n)$ in $\mathcal{Z}(\sigma)\cap \mathcal{Z}(\lambda)$, then $\Sigma^0$ is a smooth submanifold of $M$ with codimension $2$.
\end{prop}
\begin{proof}
	From Eqs.~\eqref{RelationISigma} and \eqref{RelationXLambda} one has $D_A\lambda= 2\Phi_{AB}X^B$ and $D_A\sigma= \Phi_{AB}I^B$. When evaluated at points of $\mathcal{Z}(\sigma)\cap\mathcal{Z}(\lambda)$ these give
	\begin{align}
	2\Phi_{AB}X^B &= 0 \,Y_A + \nabla_a\lambda \,Z^a_{A} &  \Phi_{AB}I^B &= 0 \,Y_A + \nabla_a\sigma \,Z^a_{A}.
	\end{align}
	To conclude the proof we need to show that there are no points of $\Sigma^0$ where $\nabla_a\lambda \propto \nabla_a\sigma$. However, at such points we would have for some density $f$:
\begin{equation}
	2\Phi_{AB}X^B + f \Phi_{AB}I^B =0.
	\end{equation}
	Since $\Phi_{AB}$ is invertible by hypothesis, this would mean $2X^B + f I^B =0$ which is excluded by definition of $\Sigma^0$.
\end{proof}

\subsection{Derivation of Theorem \ref{Thrm: Curved orbits decomposition: quotient}}\label{Proof3point2}

We will first need the following
\begin{lemm} \label{CovariantDerivativesDirectionn}
	In a generic scale $\nabla$, one has
	\begin{align}
	n^a \nabla_a \sigma &= - \rho \sigma \\
	n^a \nabla_a \lambda &= 2\sigma - 2\lambda \rho
	\end{align}	
\end{lemm}
\begin{proof}
	Since $I$ is null, it follows from Eq.~\eqref{RelationISigma} that: $I^AD_A\sigma=0$. In a generic scale $\nabla$ we have:
	\[ I^A= n^aW_a^A+\rho X^A, \quad D_A\sigma=\sigma Y_A + \nabla_a\sigma Z^a_A,\]
	and therefore \[I^AD_A\sigma = n^a\nabla_a\sigma + \rho\sigma =0.\]
	Similarly from Eq.~\eqref{RelationXLambda}, and by definition of $\sigma$, we have:
	\[ 2\sigma=I^BD_B\lambda= n^a\nabla_a\lambda + 2\rho\lambda. \]
\end{proof}

The weight $-1$ vector field $n^a$ determines a distinguished family of unparametrised curves on $\widetilde{M}$. These are precisely the integral curves of any vector field $N^a$ obtained from $n^a$ by multiplying by a nowhere vanishing density $\nu$ of weight $1$ on $\widetilde{M}$; only the parametrisation depends on the choice of $\nu$. Since $\widetilde{M}$ is compact, any choice of $N^a$ leads to a complete vector field on $\widetilde{M}$ which restricts to a complete vector field on $M$.
Lemma~\ref{CovariantDerivativesDirectionn} shows that $n^a$ is tangent to the submanifolds $\Sigma$ and $\Sigma^0$ from which it follows immediately:
\begin{coro}
	For any choice of nowhere vanishing density $\nu$ of weight $1$ on $\tilde{M}$, the action of $(\R,+)$ given by the global flow of the vector field $N=\nu n$ on $N$, preserves the decomposition~\eqref{Decomposition of the manifold}.
\end{coro}

\begin{defi}
	Define $\Q$ to be the quotient space $M/\R$, equipped with the quotient topology, where the action of $\R$ is provided by the flow of any $N=\nu n$ and $\nu$ is a nowhere vanishing scale on $M$ that has a smooth extension to $\widetilde{M}$.

Let $\pi: M \rightarrow \Q$ be the natural projection, we define the distinguished open subset $\Q_\mathcal{O}=\pi(\mathcal{O})$ of $\Q$ ($\pi$ is an open map). Since $\mathcal{O}$ is stable under the action of $\R$ this can also be defined as the quotient space $\mathcal{O}/\R$ where the action of $\R$ is that of the flow of $N^\sigma=\sigma n$ equipped with the quotient topology.
\end{defi}
 We will now show that $\Q_\mathcal{O}$ has a natural manifold structure.
\begin{theo}\label{TheoFibreBundleInside BIS}
	Let $\pi: \mathcal{O} \rightarrow \Q_\mathcal{O}$ be the canonical projection. Then there is a unique smooth structure on $\Q_\mathcal{O}$ such that $\pi$ has the structure of a (right) $\R$-principal fibre bundle. The zero set of the function $l:= \sigma^{-2}\lambda : \mathcal{O} \to \mathbb{R}$ defines a global section and yields a preferred global diffeomorphism
	\begin{equation*}
	\begin{array}{ccc}
	\mathcal{O} & \simeq &\mathbb{R} \times \Q_\mathcal{O}\\
	x & \mapsto & ( \frac{1}{2}l(x) , \pi(x) )
	\end{array}.
	\end{equation*}
In particular, $\omega_a=\frac{1}{2} \nabla_a l$ is a canonical principal connection.
\end{theo}
\begin{proof}[Proof of Theorem~\ref{TheoFibreBundleInside BIS}] We apply the characterisation of principle fibre bundles~\cite[Theorem~\textbf{4}.2.4]{Sharpe:1997aa}. It suffices to show that action is free and proper.
	We shall work in the distinguished scale determined by $\sigma$ on $\mathcal{O}$. We define $l=\sigma^{-2}\lambda$, and let $\gamma$ denote an arbitrary integral curve of $N=\sigma n$, then Eq.~\eqref{sigma scale => N^a nabla_a l= 2} leads to:
	\begin{equation}\label{LambdaAlongCurve BIS}
	 \frac{\dd}{\dd t}\lambda(\gamma(t))= (N^a\nabla_a l)|_{\gamma(t)}=2.
	\end{equation}
	It follows that $l=2t +l_0$ along the curve $N$, consequently $l$ is strictly monotonous and unbounded along such a curve.
	Thus the action is:
	\begin{itemize}
		\item free: the isotropy subgroup of any point $x$ is either $\R$ (i.e. $\forall t \in \R, \phi^N_t(x)=x$) or $T\mathbb{Z}$ for some $T\in \R$. The first case is excluded as $n^a$ does not vanish on $\mathcal{O}$. If it was of the form $T\mathbb{Z}$ with $T>0$ then the curve $t\mapsto \phi^N_t(x)$ would be periodic, but this is excluded because $l$ is strictly monotonous along such a curve. Hence, $T=0$.
		\item proper: Let $K$ be a compact subset of $M\setminus \Sigma$, then, $l$, being continuous, is bounded above and below on $K$ say: $\forall x \in K, m \leq l(x)\leq M$.
		Let $x \in K$, then:
		\[  l(\phi^N_t(x)) =  2 t + l(x) \geq 2t + m. \]
		So for any $x\in K$, when $t \geq (\frac{M-m +1}{2})$ $\phi^N_t(x)$ has definitively left $K$.
		This implies that $\{t\in \R,  t K \cap K \neq \emptyset\}$ is compact.
	\end{itemize}
Finally, Eq.~\eqref{LambdaAlongCurve BIS} implies that $\frac{1}{2}l(\phi^N_t(p))=\frac{1}{2}l(p) + t$, which shows that  $\frac{1}{2}l$ provides a canonical trivialisation of $\mathcal{O}$.
\end{proof}

\begin{rema}
For future computations it will be useful to note that $N$ is by definition a fundamental vector field on $\mathcal{O}$. 
Recall that if $P$ is a principal $G$ bundle then the fundamental vector fields $X^*$ associated with $X\in \mathfrak{g}$ is defined at $p\in P$ by
\[ A^*_p=\left.\frac{\dd}{\dd t}(p \exp(Xt))\right|_{t=0}, \]
in the simple case we consider of the abelian Lie algebra $\R$, the exponential map satisfies $\exp(1t)=t$ therefore:
\[ 1^*_x = \left. \frac{\dd}{\dd t}(p\cdot t)\right|_{t=0}=  \left. \frac{\dd}{\dd t}\phi^N_t(x)\right|_{t=0}=N_x. \]
\end{rema}

It is not clear that $\Q$ itself can be equipped with a smooth structure, as the behaviour of $n$ near $\Sigma$ (to which it is tangent) is unclear. However, to simplify our discussion we shall make the assumption that $\Q$ can be equipped with a smooth structure such that $\pi : M\rightarrow \Q $ is a fibre bundle with fibre $\R$.

We conclude this section by recording the following useful result.
\begin{lemm} \label{LieDerivativeofg}
	Let $\nu$ be any nowhere vanishing scale and $N^a = \nu n^a$, $g^{ab} = \nu^2 \zeta^{ab}$ then
	\begin{equation}
	\mathcal{L}_{N} g^{ab} = 2 N^{(a} \lambda^{b)} + 2 \rho g^{ab},
	\end{equation}
where $\lambda^{a}$ and  $\rho$  are respectively defined by equations \eqref{DH=0}  and \eqref{IParallel}.
\end{lemm}
\begin{proof}
	For any vector field $X$, we have the identity
	\[ \mathcal{L}_{X}g^{ab} = \nabla_X g^{ab} - 2 \nabla_c X^{(a} g^{b)c}. \]
	Taking $X^a = N^a$  and  making use of \eqref{DH=0} and  \eqref{IParallel} one obtains
	\begin{equation}
	\mathcal{L}_{N} g^{ab} = 2 N^{(a} \lambda^{b)}  + 2 \rho g^{ab}.
	\end{equation}
\end{proof}

\section{Geometry of the open orbit $\mathcal{O}$}\label{section: Geometry of the open orbit O}

In this section we restrict ourselves to the curved orbit $\mathcal{O}$ on which $\sigma$ is nowhere vanishing. We can therefore form the undensitised bi-vector, vector field and function
\begin{align}
g^{ab} &= \sigma^2 \zeta^{ab}, & N^a &= \sigma n^a&  l&=\sigma^{-2}\lambda.
\end{align}

The scale $\nabla$ determined by $\sigma$ (i.e. such that $\nabla\sigma=0$) has the following important properties:
\begin{lemm}\label{Lemma: sigma scale => identities} 
	In the scale $\nabla$ determined by $\nabla \sigma =0$, we have:
	\begin{align}
	\nabla_c n^a &= -\delta^a_c \rho =0, \label{sigma scale => rho=0} \\
	\lambda n^a&=-\frac{1}{2}\sigma \zeta^{ab}\nabla_b\lambda,  \label{sigma scale => n=grad(l)} \\ 
	\nabla_c \zeta^{ab}&=-2\sigma^{-1}\delta^{(a}_{c}n^{b)} \label{sigma scale => lambda^a =n^a}.
	\end{align}
\end{lemm}
\begin{proof}
	Equations \eqref{sigma scale => rho=0} follows from Eqs.~\eqref{IParallel}, \eqref{RelationISigma}, $I^2=0$ and the definition of the scale:
	\begin{equation*}
	\rho = I^AY_A= I^A(\sigma^{-1}D_A\sigma) =\sigma^{-1}I^2 =0.
	\end{equation*} 
	To prove the second relation $\eqref{sigma scale => n=grad(l)}$ we remark that from Eq.~\eqref{RelationXLambda} one has: $X^A=\frac{1}{2}H^{AB}D_B\lambda$. Contracting with $Z_A^a$ then gives
	\[0=\frac{1}{2}Z_A^aH^{AB}D_B\lambda= \frac{1}{2}\zeta^{ab}\nabla_b\lambda + Z_A^a\underbrace{H^{AB}Y_B}_{\sigma^{-1}H^{AB}D_B\sigma}\lambda=\frac{1}{2}\zeta^{ab}\nabla_b\lambda +\sigma^{-1}n^a\lambda,\]
	where we have used Eq.~\eqref{RelationISigma}. Finally, the last relation $\eqref{sigma scale => lambda^a =n^a}$  comes from:
	\[\nabla_c \zeta^{ab}= \nabla_c Z_A^aH^{AB}Z_B^b=-\delta^{a}_{c}Y_AH^{AB}Z_B^b-\delta^{b}_{c}Z_B^bH^{AB}Y_B=-2\sigma^{-1}\delta^{(a}_{c}n^{b)}.\]
\end{proof}

\begin{rema}Since $\sigma$ is covariantly constant for the connection $\nabla$, Eqs.~\eqref{sigma scale => rho=0},\eqref{sigma scale => n=grad(l)}, and~\eqref{sigma scale => lambda^a =n^a} translate directly in terms of the quantities $g,N,f$:
\[\nabla_c N^a =0, \quad N^a= -\frac{1}{2}l^{-1}g^{ab}\nabla_b l, \quad \nabla_c g^{ab}=-2\delta_{c}^{(a}N^{b)}.\] \end{rema}

Equation \eqref{sigma scale => rho=0} means that the integral lines of $N^a$ are geodesically parametrised for $\nabla$. In fact, combining this with Lemma~\ref{CovariantDerivativesDirectionn}, one sees that
\begin{equation}\label{sigma scale => N^a nabla_a l= 2}
N^a \nabla_a l= 2
\end{equation}
in other terms $l$ coincides with this parametrisation. This however fails to be a gradient flow for $g^{ab}$: Eq.~\eqref{sigma scale => n=grad(l)} recovers that $g^{ab}$ is degenerate along $l=0$, and one rather has $N^a = -\frac{1}{2} g^{ab}(l^{-1} \nabla_b l)$. The last identity \eqref{sigma scale => lambda^a =n^a} is crucial to derive the following canonical form:
\begin{prop}\label{Proposition: canonical form on O}
	\begin{equation}
	g^{ab} = -l N^a N^b + \tilde{h}^{ab}
	\end{equation}
	where $\tilde{h}^{ab}\nabla_b l=0$ and $\mathcal{L}_{N}\tilde{h}^{ab}=0$. 
\end{prop}
\begin{proof}
	We recall from Lemma~\ref{LieDerivativeofg} that in a generic scale $\nu$
	\begin{equation}
	\mathcal{L}_{N} g^{ab} = 2 N^{(a} \lambda^{b)} + 2 \rho g^{ab},
	\end{equation}
	where $\lambda^a$ is defined by the relation $\nabla_c \zeta^{ab} - 2 \delta^{(a}_c \lambda^{b)} =0$. 
	When evaluated in the scale $\sigma$ this becomes, by making use of \eqref{sigma scale => rho=0} and \eqref{sigma scale => lambda^a =n^a},
	\begin{equation}
	\mathcal{L}_{N} g^{ab} = -2N^{a} N^b.
	\end{equation}
	Since $N^a \nabla_a l=2$, this is equivalent to $\mathcal{L}_{N}(  g^{ab} + l N^{a} N^b)=0$ and therefore one can write
	\begin{equation}
	g^{ab} = -l N^{a} N^b + \tilde{h}^{ab}
	\end{equation}
	with $\mathcal{L}_{N} \tilde{h}^{ab}=0$. We conclude by computing $\tilde{h}^{ab} \nabla_b l$
	\begin{align}
	\tilde{h}^{ab} \nabla_b l &= g^{ab}\nabla_b l +l N^{a} N^b\nabla_b l\\
	&= -2 l N^a + 2 l N^a =0
	\end{align}
	where we made use of \eqref{sigma scale => n=grad(l)}.
\end{proof}
Using Lemma~\ref{Lemma: sigma scale => identities} we have:
\begin{prop}\label{Proposition: einstein eq sur g}
	At points of $\mathcal{O}$ where $g^{ab}$ is invertible, the corresponding metric is Ricci-flat. 
\end{prop}
\begin{rema}
Note that the Ricci flat metric $\sigma^2\zeta^{ab}$ of this proposition differs from the Einstein metric with a non-vanishing scalar curvature $\lambda \zeta^{ab}$ that is given by the projective compactification of order 2 as in \cite{Cap:2014aa}.
\end{rema}

\begin{proof}
Calculating in the scale $\nabla$, \[\begin{aligned} 0= D_B I^A =D_B(H^{AC}D_C\sigma)=H^{AC}D_B (\sigma Y_C),
\end{aligned} \]
since $H^{AB}$ is non-degenerate we conclude that $D_B(\sigma Y_C)=0$, therefore,
\[0=D_B(\sigma Y_C)=(D_B\sigma) Y_C + \sigma(D_B Y_C)=\sigma P_{bc}Z^b_BZ^c_C, \]
hence, on $\mathcal{O}$, $P_{bc}=0$.

Define, when $l\neq 0$, a torsion-free connection $\tilde{\nabla}$ by:
\begin{equation}\label{RelationScaleSigmaLeviCivita} \tilde{\nabla}_c \xi^b = \nabla_c \xi^b +N^bg_{cd}\xi^d,\end{equation}
for any vector field $\xi^b$.
\begin{rema}
Note that this is \emph{not} a projective change of connection.
\end{rema}
It is straightforward to check that $\tilde{\nabla}$ is the Levi-Civita connection of $g$. Let us calculate the Ricci tensor $\tilde{R}_{ab}$ of $\tilde{\nabla}$ with respect to that of $\nabla$ which we denote by $R_{ab}$. In fact, $R_{ab}$ vanishes since $\sigma$ is a special scale and $R_{ab}=(n-1)P_{ab}=0$.
We have:
\begin{equation}
\tilde{R}_{ab}= \underbrace{R_{ab}}_{=0} + \underbrace{N^c\nabla_c g_{ab}}_{\frac{1}{2}\Upsilon_a\Upsilon_b} + \frac{1}{2}\nabla_a\Upsilon_b -\frac{1}{4}\Upsilon_a\Upsilon_b-\frac{1}{2}\underbrace{(N^c\Upsilon_c)}_{2\sigma^2\lambda^{-1}}g_{ab}
\end{equation}
Thus:
\begin{equation}
\tilde{R}_{ab}= \frac{1}{2}\nabla_a\Upsilon_b +\frac{1}{4}\Upsilon_a\Upsilon_b-\sigma^2\lambda^{-1}g_{ab}
\end{equation}
where we have introduced: $\Upsilon=\lambda^{-1}\nabla \lambda$.
Using $H^{AB}\Phi_{BC}=\delta^A_C$ and Equation~\eqref{InverseTractorMetric} one finds that:
\[ \lambda^{-1}\sigma^2 g_{ab}= \frac{1}{2} \left(\lambda^{-1}\nabla_a\nabla_b\lambda - \frac{1}{2}\lambda^{-2}\nabla_a \lambda\nabla_b \lambda \right)=\frac{1}{2}\left(\nabla_a \Upsilon_b + \frac{1}{2}\Upsilon_a\Upsilon_b \right),\]
which shows that:
\[ \tilde{R}_{ab}= 0.\]
\end{proof}

Proposition~\ref{Proposition: canonical form on O} suggests that $\tilde{h}^{ab}$ might induce a smooth tensor on the quotient $\Q_{\mathcal{O}}$ of $\mathcal{O}$ by the integral lines of $N^a$. In turn, Eqs~\eqref{sigma scale => lambda^a =n^a} and~\eqref{RelationScaleSigmaLeviCivita} suggest that the connection $\nabla$ factors to the quotient and this factorisation will yield the Levi-Civita connection of such a tensor field; the following results confirm this intuition.

\begin{prop}\label{Proposition: invertible metric on Q_O}
	$g^{ab}$ induces a bilinear form $h^{ab}$ on $\Q_{\mathcal{O}}$ which is everywhere non-degenerate.
\end{prop}
\begin{proof}
	Let $\alpha,\beta$ be $1$-form fields on $\Q_{\mathcal{O}}$ and consider the smooth function on $\mathcal{O}$ defined by $g(\pi^*\alpha,\pi^*\beta)$, Lemma~\ref{Proposition: canonical form on O} implies that $g(\pi^*\alpha,\pi^*\beta)= \tilde{h}(\pi^*\alpha,\pi^*\beta)$ is constant on the fibres of $\mathcal{O} \to \Q_{\mathcal{O}}$ and so factorises to a well-defined smooth map on $\Q_{\mathcal{O}}$, that we denote by $h(\alpha,\beta)$. It is clearly symmetric and tensorial so determines a bilinear form on $T^*\Q_{\mathcal{O}}$.

	Now $\zeta^{ab}$ is degenerate exactly at points where $\lambda =0$ and so $g^{ab} = \sigma^{2}\zeta^{ab}$ is invertible at any point where $l= \sigma^{-2}\lambda \neq 0$. However, at any point of $\mathcal{O}$, one has the decomposition $T^*\mathcal{O} = \text{Span}(\nabla_a l) \oplus \pi^* (T^*\Q_{\mathcal{O}})$ and by Lemma~\ref{Proposition: canonical form on O} the corresponding decomposition of $g^{ab}$ is $\text{diag}(-l, h^{ab})$. Therefore $g^{ab}$ can only be invertible at points where $l\neq0$ if $h^{ab}$ is.	
\end{proof}

\begin{prop}\label{Connection invariant under flow O}
The connection $\nabla$ is invariant under the flow of $N$.
\end{prop}
\begin{proof}
First we consider instead the principal connection $\alpha$ determined by $\nabla$ on the frame bundle $P^1(\mathcal{O})$. 
The $1$-parameter group of diffeomorphisms generated by $N$ induces a $1$-parameter group of diffeomorphisms on $P^1(\mathcal{O})$. Since $N$ is parallel, the infinitesimal generator $\tilde{N}$ of the induced $1$-parameter group of diffeomorphisms\footnote{This is also referred to as the complete or natural lift of $N$, see for instance~\cite{Cordero:1983aa}.} is precisely the \emph{horizontal lift} of $N$ with respect to the principal connection $\alpha$. This can be seen, for instance, by a swift computation in a canonical coordinate system of $P^1(\mathcal{O})$.
Since the Lie derivative satisfies the Leibniz rule we have:
\[ (\mathcal{L}_{\tilde{N}}\alpha)(\tilde{X})=\tilde{N}(\alpha(\tilde{X})) - \alpha([\tilde{N},\tilde{X}]). \]
If $\tilde{X}=X^*, X\in \mathfrak{gl}(n+1,\R)$ is a fundamental vector field, this vanishes entirely; indeed: $\alpha(X^*)=X$ is constant and the Lie bracket between a fundamental field and a horizontal vector field vanishes. 
Hence, $\mathcal{L}_{\tilde{N}}\alpha$ is a horizontal and equivariant $\mathfrak{gl}(n+1,\R)$ valued 1-form.
In fact, restricting to the case where $\tilde{X}$ is horizontal ($\alpha(\tilde{X})=0$), we see that it reduces to
\[(\mathcal{L}_{\tilde{N}}\alpha)(\tilde{X})=-\alpha([\tilde{N},\tilde{X}])=\Omega(\tilde{N},\tilde{X}), \]
where $\Omega$ is the curvature $2$-form.
It follows that we should compute: $N^aR_{ab\phantom{c}d}^{\phantom{ab}c}.$ 

Using Lemma~\ref{Lemma: sigma scale => identities}, we see that although $\nabla$ is not the Levi-Civita connection of $g$, it does satisfy: $\nabla_d \nabla_c g^{ab}=0,$ which leads to the following symmetry of the curvature tensor:
\[ R^{abcd}=-R^{abdc},\]
where all indices are raised with $g^{ab}$. This implies that $R^{abcd}$ has all the symmetries of the curvature tensor of the Levi-Civita connection of a metric. In particular, at points where $l\neq 0$,
\[N^aR_{a}^{\phantom{a}bcd}=-\frac{1}{2}l^{-1}\nabla_a l R^{abcd}=\frac{1}{2}l^{-1}\nabla_al R^{cdba}=-R^{cdb}_{\phantom{cdb}a}N^a. \]
However, since $N^a$ is parallel, we must have: $R_{ab\phantom{c}d}^{\phantom{ab}c}N^d=0$, thus at points where $l\neq 0$ (at which $g^{ab}$ is non-degenerate):
\[ N^aR_{ab\phantom{c}d}^{\phantom{ab}c}=0,\]
but these points are dense in $\mathcal{O}$, so the above identity extends to $\mathcal{O}$ by continuity.
Consequently:
\[ \mathcal{L}_{\tilde{N}}\alpha=0.\]
\end{proof}

Before we state the immediate corollary, we first record a consequence of the final steps of our computation:
\begin{lemm}\label{ContractionNandWeyl}
At every point of $M$:
\[n^a\Weyl{a}{b}{c}{d}=0\]
\end{lemm}
\begin{proof}
Both the Weyl tensor $\Weyl{a}{b}{c}{d}$ and $n^a$ are projectively invariant, hence this is a projectively invariant statement and it suffices to check in one particular scale. On $\mathcal{O}$, $\sigma$ is a scale and we have seen above that the Riemann tensor of the connection $\nabla^\sigma$ satisfies:
\[N^a\Riem{a}{b}{c}{d}=\sigma n^a\Riem{a}{b}{c}{d}=0, \]
therefore $n^a\Riem{a}{b}{c}{d}=0$ where $\sigma\neq 0$.
However, we saw in the proof of Proposition~\ref{Proposition: einstein eq sur g} that $P_{bc}=0$ in this \emph{special} scale, i.e. $\beta_{ab}=0$ too. Hence: $\Riem{a}{b}{c}{d}=\Weyl{a}{b}{c}{d}$ and thus the result holds on $\mathcal{O}$, since $\mathcal{O}$ is dense in $M$ it extends to $M$ by continuity.
\end{proof}

The immediate corollary of Proposition~\ref{Connection invariant under flow O} is that :
\begin{coro}\label{TheoFactorisationNablaSigma}\label{Proposition: connection on Q}
$\nabla$ induces a torsion-free connection $\bar{\nabla}$ on the quotient $\Q_{\mathcal{O}}$.
\end{coro}
\begin{proof}
Let $\bar{X},\bar{Y}$ be vector fields on $\Q_{\mathcal{O}}$, and $X,Y$ their horizontal lifts to vector fields $\mathcal{O}$ using the principal connection $\omega$ in Theorem~\ref{TheoFibreBundleInside BIS} and define:
\[\bar{\nabla}_X Y = \pi_*(\nabla_{\overline{X}}\overline{Y}) \]
By the previous proposition, this is independent of choice on the fibre and so is a well-defined vector field on $Q$. This is clearly linear in $X$ and satisfies the Leibniz rule in $Y$ and therefore is a connection.
It is torsion-free as:
\[ \bar{\nabla}_{\bar{Y}}{\bar{X}}-\bar{\nabla}_{\bar{X}}{\bar{Y}}=\pi_*(\nabla_X Y -\nabla_Y X)= \pi_*([X,Y])=[\pi_*(X),\pi_*(Y)]=[\bar{X},\bar{Y}], \]
by naturality of the pushforward.
\end{proof}

\begin{prop}\label{Proposition: einstein eq sur h}
The induced connection $\bar{\nabla}$ on $\Q_{\mathcal{O}}$ is the Levi-Civita connection of the form $h$ defined in Proposition~\ref{Proposition: invertible metric on Q_O}; furthermore the Lorentzian geometry $(Q,h_{ab})$ is Ricci-flat.
\end{prop}

\begin{proof}
If $\bar{X}$ is a vector field on $\Q_{\mathcal{O}}$, $X$ its horizontal lift to a vector field on $\mathcal{O}$, and $\alpha, \beta$ are 1-forms on $\Q_{\mathcal{O}}$, then at every point $x\in \mathcal{O}$ we have:
\[ \begin{aligned} (\bar{\nabla}_{\bar{X}} h)(\alpha,\beta)(\pi(x))&= \nabla_{\bar{X}}g(\pi^*\alpha, \pi^*\beta)(x)\\&=-(\pi^*\alpha)(N)(\pi^*\beta)(X)(x)-(\pi^*\beta)(N)(\pi^*\alpha)(X)\\&=0. \end{aligned}\]
This shows that $\bar{\nabla}$ is the Levi-Civita connection for $h$.

Now, if $\bar{X},\bar{Y}$ and $\bar{Z}$ are vector fields on $\Q_\mathcal{O}$ and $X,Y,Z$ their horizontal lifts to $\mathcal{O}$ then, noting that $\nabla_a l[X,Y]^a=0$, and using that $N$ is parallel for $\nabla$, we have:
\[\bar{R}(\bar{X},\bar{Y})\bar{Z}= \pi_*(R(X,Y)Z). \]

If $(\bar{E}_i)$ is an arbitrary local frame of $T\Q_{\mathcal{O}}$ then lifting them horizontally to fields on $\mathcal{O}$, we obtain a local frame $(E_i,N)$ of $TM$, with dual frame: $(\omega^i, \frac{1}{2}\dd l)$. Now:
\[ \textrm{Ric}(Y,Z)= \sum_i \omega^i(R(E_i, Y)Z) + \frac{1}{2}\dd l(R(N,Y)Z)= \sum_i \omega^i(R(E_i, Y)Z).  \]
By definition, $\omega^i(E_j)=\delta^i_j$, $\omega^i(N)=0$ therefore: $\mathcal{L}_N \omega^i(E_j) = -\omega^i([N,E_j])=0$. Thus, the $\omega^i$ are the pullbacks of a local frame $\bar{\omega}^i$ of $\Q_\mathcal{O}$ and it follows that:
\[ 0=\textrm{Ric}(Y,Z)=\overline{\textrm{Ric}}(\tilde{Y},\tilde{Z}).\]

We can also check this in terms of the global trivialisation given in Theorem~\ref{TheoFibreBundleInside BIS}:

	\begin{equation}
	\begin{array}{ccc}
	T\mathcal{O} & \to & \text{Span}(N) \oplus T\Q_\mathcal{O}\\
	v^a & \mapsto &( \frac{1}{2}(v^b dl_b) N^a , v^b \nabla_bx^{\mu})
	\end{array}
	\end{equation}
	the decomposition of the tangent bundle given by $l$. At any point where  $\lambda \neq 0$ the inverse of $g^{ab}$ is, in this decomposition,
	\begin{equation}
	g_{ab} = \begin{pmatrix}-\frac{1}{l} & 0 \\ 0 &  h_{\mu\nu}\end{pmatrix}.
	\end{equation}
	The local connection form of the Levi-Civita connection in coordinates $(x^0=\frac{1}{2} l, x^\mu)$ adapted the global trivialisation is then given by
	\begin{equation}\label{Levi-civita in the scale sigma}
	(\omega^i_{\, j})_c= \nabla_c x^0\begin{pmatrix}
	-\frac{1}{2l} & 0 \\ 0 & 0
	\end{pmatrix} +  \nabla_c x^\rho \begin{pmatrix}
	0 & 0 \\ 0 & \gamma^{\mu}_{\rho\nu}
	\end{pmatrix}
	\end{equation}
	where $\gamma^{\mu}_{\rho\nu}$ are the Christoffel symbols of the Levi-Civita connection of $h$. The local curvature form is
	\begin{equation}\label{Riemann in the scale sigma}
	(F^i{}_{j})_{cd}= \nabla_cx^{\rho}\wedge\nabla_d x^{\sigma} \begin{pmatrix}
	0 & 0 \\ 0 & R^{\mu}{}_{\nu\rho\sigma}
	\end{pmatrix}
	\end{equation}
	where $R^{\mu}{}_{\nu\rho\sigma}$ is the Riemann tensor of $h$. Since by Proposition~\ref{Proposition: einstein eq sur g} $g_{ab}$ is Ricci-flat it immediately follows that $h_{ab}$ is also.
\end{proof}

\section{Geometry of $\Q$}\label{section: Geometry of Q}

We recall that $M := \widetilde{M}\setminus \mathcal{Z}(n)$ and assume, as in Theorem \ref{Thrm: Curved orbits decomposition: quotient}, that its quotient $\Q$ by the flow of $n^a$ can be equipped with a smooth manifold structure. It then decomposes as
\begin{equation}
\Q = \Q_\mathcal{O} \cup \mathcal{H}^- \cup \mathcal{H}^0 \cup \mathcal{H}^+ .
\end{equation}

 From the previous section, in particular Proposition \ref{Proposition: invertible metric on Q_O} and Proposition  \ref{Proposition: einstein eq sur h}, we learned that $\Q_\mathcal{O}$ is equipped with a Ricci flat Lorentzian metric. We will now show that this metric is projectively compact of order 1 with $\mathcal{H}^- \cup \mathcal{H}^0 \cup \mathcal{H}^+$ forming the projective boundary.
 
In order to prove this we will proceed in 3 steps. First, we will discuss how to identify densities and tractor bundles on $Q$ with subbundles of bundles on $M$ via pullback.
This first step can be thought as an infinitesimal version of the relations between the homogeneous models as in Figure~\ref{Figure: homogeneous space}. In particular, since $H^{AB}$ and $I_A$ are covariantly constant on the whole of $M$ they will define corresponding tractor fields $\bar{H}^{AB}$ and $\bar{I}_A$ on  $\Q$. As a result of the quotient by $I^A$ which is necessary for the identification, $\bar{H}^{AB}$ then is degenerate of kernel $\bar{I}_A$.  
Secondly, we will show that the projective tractor connection $\nabla$ of $ M$ induces a normal projective tractor connection $\bar{\nabla}$ on $\Q$ and that $\bar{\nabla}_c\bar{H}^{AB}=0$, $\bar{\nabla}_c\bar{I}_{A}=0$. This will imply by  \cite{Cap:2014ab,Flood:2018aa} that $\Q$ is equipped with a Ricci flat metric projectively compact of order 1. Finally we will show that on $\Q_\mathcal{O}$ this metric coincides with the Ricci flat metric resulting from the previous section.

\subsection{Densities and tractors of $\Q$}
We will first show how to relate densities on $M$ with densities on $\Q$. Extending this at the level of $1$-jets will yield the corresponding identification between projective tractors $\T_M$ on $M$ with projective tractors $\T_\Q$ of $\Q$.

\subsubsection{Densities on $\Q$}

In the sequel, we shall assume that the flow of $n^a \in \Gamma\left(TM(-1)\right)$ on  $M :=  \widetilde{M} \setminus \mathcal{Z}(n)$ defines a fibre bundle
\begin{equation}\label{M -> Q}
\pi :\begin{array}{ccc}
M & \longrightarrow &\Q
\end{array}
\end{equation}
with fibres given by the integral lines, \textrm{and general fibre $\R$}.

Let $\mathcal{E}_M(w)$ and $\mathcal{E}_\Q(w)$ be projective densities of $M$ and $\Q$ respectively.
\begin{prop}\label{Proposition: Isomorphism between densities}
	There is a canonical isomorphism
\begin{equation}\label{Isomorphism between densities}
\pi^*\E_{\Q}(w)\simeq  \E_M(w).
\end{equation}
\end{prop}
\begin{proof}
By definition of projective densities (see Section~\ref{Projective Tractor Calculus}), the bundle $\mathcal{E}_{M}(w)$ is canonically isomorphic to $\big(\textnormal{Vol}(M)\big)^{-\frac{w}{n+2}}$ where $\textnormal{Vol}(M) = \big(\Lambda^{n+1}T^*M\otimes\Lambda^{n+1} T^*M\big)^{\frac{1}{2}}$ is the bundle of volume densities of $M$. 

Therefore, there always exists a canonical nowhere vanishing section $\varepsilon^2_{a_1\dots a_{n+1}, b_1 \dots b_{n+1}}$ of $  [(\Lambda^{n+1}T^*M)\otimes(\Lambda^{n+1} T^*M)](2(n+2))$ realising the canonical isomorphism
\begin{equation}
\varepsilon^2_{a_1\dots a_{n+1}, b_1 \dots b_{n+1}} :
\begin{array}{ccc}
\E_M(-2(n+2)) &\simeq& \Lambda^{n+1}T^*M \otimes \Lambda^{n+1}T^*M
\end{array}
\end{equation}
We will note $TM/n$ the fibre bundle over $M=\widetilde{M}\setminus \mathcal{Z}(n)$ obtained by taking the point-wise quotient of $TM$ by the span of $s n^a$ where $s$ is any \enquote{true scale}, i.e. nowhere vanishing density of weight $1$; the quotient is independent of the choice of $s$. By construction, the annihilator $\mathcal{N}^*$ of $n^a$ in $T^*M$ is identified with $(TM/n)^*$ as well as with the pull-back bundle $\pi^*(T^*Q)$. It follows that:

\begin{equation}\label{proof: nu volume definition}
\nu_{a_1\dots a_n, b_1\dots b_{n}}:= \varepsilon^2_{aa_1\dots a_{n}, bb_1 \dots b_{n}}n^an^b,
\end{equation}
is identified with a nowhere vanishing section of $\pi^*\big(\Lambda^n T^*Q \otimes \Lambda^n T^*Q \big)\otimes \E_M\big(2(n+1)\big)$ and thus realises an isomorphism
\begin{equation}\label{Proof: intermediate density isomorphism}
\mathcal{E}_M\big(-2(n +1)\big) \simeq \pi^*\big(\Lambda^n T^*Q \otimes \Lambda^n T^*Q \big) \simeq \pi^*\mathcal{E}_{\Q}(-2(n+1)).
\end{equation}
Taking roots yields the desired isomorphism for arbitrary weights $w$.\end{proof}

This identification enables us to construct a natural differential operator on projective densities of $M$. Indeed, the action of the Lie derivative on $\pi^*\big(\Lambda^n T^*Q \otimes \Lambda^n T^*Q \big)$ along $n^a$ induces a differential operator $D_n$ on $\mathcal{E}_{M}(w)$ by closing the following commutative diagram:
\begin{figure}[h!!]
\centering
\begin{tikzcd}
\pi^*\left( \Lambda^n T^*Q \otimes \Lambda^n T^*Q \right)^{-\frac{w}{2(n+1)}} \arrow[d,"\simeqd"]\arrow[r,"\mathcal{L}_n"] &   \pi^*\left( \Lambda^n T^*Q \otimes \Lambda^n T^*Q \right)^{-\frac{w}{2(n+1)}}(-1)\arrow[d,"\simeqd"]\\
 \mathcal{E}_{M}(w) \arrow[r, dashed, "D_n"]&  \mathcal{E}_{M}(w-1).
\end{tikzcd}
\caption{Defining diagram of $D_n$ \label{Commutative diagram for Lie derivatives}}
\end{figure}

\begin{prop}\label{Proposition: Linear vertical differential operator on densities}
	The maps appearing in the commutative diagram in Figure~\ref{Commutative diagram for Lie derivatives} are well defined linear differential operators of order $1$. In a given projective scale, we have
	\begin{equation}\label{Vertical derivative, definition}
	\begin{array}{cccc}
	D_n : &\mathcal{E}_{M}(w) &\longrightarrow & \mathcal{E}_{M}(w-1)\\
	&\tau &\longmapsto & (\nabla_{n} +w\rho)\tau
	\end{array}
	\end{equation}
	where $\rho := -\tfrac{1}{n+1} \nabla_a n^a=Y_AI^A.$
\end{prop}

\begin{proof}
Let $s$ be an arbitrary non-vanishing density of weight $1$ on $M$, and $X^a=sn^a$. Since $\pi^*(\Lambda^nT^*Q)$ can be identified with $n$-forms on $M$ that vanish whenever at least one of its arguments is a vertical vector, we can apply the Lie derivative in the usual sense. Let $\nabla$ be the scale determined by $s$ in the projective class then for any $\alpha \in \pi^*(\Lambda^nT^*Q)$:
\[ \mathcal{L}_X\alpha= X \lrcorner \dd\alpha = \nabla_X \alpha - n s\rho \alpha= s(n^a\nabla_a\alpha -n\rho \alpha), \]
where we recall $\nabla_b n^a = - \rho \delta_b^a$.
Since this is proportional to the scale $s$, we see that this yields a well-defined linear operator $\mathcal{L}_{n} : \pi^*(\Lambda^n T^*Q) \to  \pi^*(\Lambda^n T^*Q)(-1)$.
According to the commutative diagram in Figure~\ref{Commutative diagram for Lie derivatives}, to define $D_n$ on densities of weight $-2(n+1)$, we must set:
\[\begin{aligned} D_n \tau = \nu^{-1}\mathcal{L}_n (\nu \tau)&= \nu^{-1} \left(\nabla_n (\nu \tau) - 2n\rho \nu\tau\right) \\
&=\nu^{-1} \left(\nu \nabla_n \tau -2\rho \nu \tau - 2n\rho \nu \tau \right)\\&= \nabla_n \tau - 2(n+1)\rho \tau.
\end{aligned}\]
where $\nu$ is a short-hand notation for the section \eqref{proof: nu volume definition}.
Eq. \eqref{Vertical derivative, definition} follows for densities of arbitrary weight $w$.  
\end{proof}

Proposition \ref{Proposition: Isomorphism between densities} allows us to identify projective densities $\mathcal{E}_{M}(w)$ on $M$ with pullbacks of densities on $\Q$. Since sections $\tau$ of $\pi^*\mathcal{E}_{\Q}(w)$ which are the pullback $\tau = \pi^* \bar{\tau}$ of sections of $\mathcal{E}_{\Q}(w)$ are precisely those with vanishing Lie derivative in the direction $n^a$ we obtain, from the commutative diagram~\ref{Commutative diagram for Lie derivatives}, the following characterisation of densities on $Q$:

\begin{prop}\label{Proposition: equivalence of densities}
		Densities $\bar{\tau}$ of weight $w$ on $\Q$ are canonically identified with densities (of weight $w$) $\tau$ on $M$ satisfying $D_n\tau=0$:
\begin{align}\label{Equivalence of densities on Q and M}
D_n \tau &=0,\quad\tau \in \Gamma(\mathcal{E}_{M}(w)) & \Leftrightarrow& & \tau &= \pi^* \bar{\tau}, \qquad \bar{\tau} \in \Gamma(\mathcal{E}_{\Q}(w)).
\end{align}
\end{prop}
In the righthand side of the above equivalence, and throughout the following sections of this article, we abuse notation and use the isomorphism \eqref{Isomorphism between densities} to identify sections of $\mathcal{E}_M(w)$ and $\pi_*\left(\E_{\Q}(w)\right)$, i.e. we will think of $\pi^*\bar{\tau}$ as a section of $\mathcal{E}_M(w)$.

\begin{rema}\label{ScalesWithRho=0Exist}
	This shows that when $\Q$ can be equipped with a smooth manifold structure, there is a nowhere vanishing density $s$ of weight $1$ on $M$ such that $\rho=Y_AI^A=0$ in the scale $\nabla^s$. Indeed, let $\tilde{s}$ be a nowhere vanishing scale on $\Q$; we have shown that this corresponds to a nowhere vanishing density $s  = \pi^* \tilde{s}$ such that $I^AD_As =0$, writing this last equation in the scale determined by $s$ shows that $s\rho =0 \Rightarrow \rho=0.$ 
\end{rema}
The above remark motivates the following definition:

\begin{defi}\label{AdaptedScales}
A nowhere vanishing $1$-density on $M$ that satisfies $D_n \nu =0$ will be called an \textbf{adapted scale} on $M$.
\end{defi}

These scales are useful as they guarantee the most compatibility between the structures on $M$ and $\Q$; we record below of their some fundamental properties.
\begin{lemm}\label{Lemma: PropertiesAdaptedScales}
Let $s$ be an adapted scale on $M$, and $\nabla$ the scale it determines in the projective class. Then:
\begin{equation}
\rho = Y_A I^A =0, \quad P_{ab}n^b = 0,\quad \nabla_a n^b =0.
\end{equation}
\end{lemm}

Using their characterisation and Eq. \eqref{CovariantDerivativesDirectionn}, we can immediately identify some important adapted scales on $\mathcal{O}$ and $\Sigma\setminus \{\lambda= 0\}$:
\begin{coro}\label{Corollary: Densities induced on Q}\mbox{}
	\begin{itemize}
		\item The density $\sigma := H^{AB}X_A I_B$ on $M$ induces a weight 1 density $\bar{\sigma}$ on
\begin{equation*}
\Q = \Q_\mathcal{O} \cup \underbrace{\mathcal{H}^- \cup \mathcal{H}^0 \cup \mathcal{H}^+}_{\mathcal{H}}
\end{equation*}
such that $\bar{\sigma}(x) =0 \;\Leftrightarrow\; x \in \mathcal{H}=\Sigma/\R$; in particular it is a boundary defining function for $\mathcal{H}$.

\item The density $\lambda := H^{AB}X_A X_B$ on $M$ induces a weight 2 density $\bar{\lambda}$, on
\begin{equation*}
\mathcal{H} = \mathcal{H}^- \cup \mathcal{H}^0 \cup \mathcal{H}^+
\end{equation*}
such that $\bar{\lambda}(x) =0 \;\Leftrightarrow\; x \in \mathcal{H}^0$.
	\end{itemize}
\end{coro}

\subsubsection{Tractors of $\Q$}\label{sssection: Tractors of Q}

In the previous subsection we established an equivalence \eqref{Equivalence of densities on Q and M} between densities on $\Q$ and densities on $M$ which are in the kernel of \eqref{Vertical derivative, definition}. In this section we will extend this equivalence to the corresponding tractor bundles $\mathcal{T}_{M} = \left(J^1 \mathcal{E}_{M}(1)\right)^* \rightarrow M $ and  $\mathcal{T}_{\Q} = \left(J^1 \mathcal{E}_{\Q}(1)\right)^* \rightarrow \Q$.

\begin{prop}\label{Proposition: Isomorphism of tractors on Q/M}
The isomorphism of Proposition \ref{Proposition: Isomorphism between densities} extends at the level of 1-jet and yields the canonical isomorphisms
\begin{align}
	\pi^*(\T_{\Q})^* &\simeq I^\circ \subset \T^*_{M}, & \pi^*\T_{\Q} &\simeq \T_M /I.
	\end{align}
\end{prop}
\begin{proof}
As we have seen, any section $\tau$ of $\E_M(1)$ that satisfies $D_n\tau=0$, arises as $\pi^*\bar{\tau}$ for some section $\bar{\tau}$ of $\E_{\Q}(1)$. Now $0=D_n\tau= I^AD_A \tau$ so that at each point $p\in M$, $D_A \tau$ is in $I^\circ_p$. 
Consider the map $D_A \tau \mapsto \pi^* D_A \bar{\tau}$. By construction, $D_A \bar{\tau}_{\pi(p)}$ only depends on the $1$-jet of $\bar{\tau}$ at $\pi(p)$, and similarly for $D_A \tau$ at $p$, therefore this defines a vector bundle morphism $I^\circ \to \pi^*\T_{\Q}^*$, which is injective in the fibres. Since they have the same dimension this is a vector bundle isomorphism.
The other isomorphism is obtained by duality.
\end{proof}

Working in adapted scales (see Definition~\ref{AdaptedScales}), we can determine an analogue of~Proposition \ref{Proposition: equivalence of densities} that tells us when a section of $I^\circ \subset \T_M^*$ arises as a pullback of a section of $\T_{\Q}^*$. 
First, let us point out that the splitting sections $Y$ are $W$ determined by a connection $\nabla$ in a projective class in fact only depend on the connection $\nabla$ induces on the density bundle $\mathcal{E}(1)$. Therefore, one can define $Y$ and $W$ even \emph{in the absence of a projective structure} by simply choosing a connection on $\mathcal{E}(1)$. In particular, for every nowhere vanishing density $\nu$, there is always a connection on $\mathcal{E}(1)$ for which it is parallel.
\begin{coro}\label{Proposition: explicit tractor isomorphism}
A co-tractor field $V_A$ on $M$ can be written as the pullback $V_A =\pi^*\bar{V}_A$ of a tractor field $\bar{V}_A$ on $Q$ if and only if:
\begin{align*} 
n^a\nabla_a V_A &=0 &&\textrm{ and } &I^AV_A&=0.
\end{align*}
Let $\nu = \pi^*\bar{\nu}$ be an adapted scale on $M$, by the preceding remark, $\nu$ and $\bar{\nu}$ determine respectively connections on $\E_\Q(1)$, $\E_{M}(1)$ and splitting operators $\overline{Y},\overline{W}$ and $Y,W$  on\footnote{One has $Y_AI^A=\nu^{-1}I^AD_A \nu=\nu^{-1}D_n \nu=0$ and these are thus splitting operators on $I^{\perp} \subset \T_{M}^*$.} $(\T_\Q)^*$ and $(\T_{M})^*$; this is summarised by Figure~\ref{Figure: Compatible splittings on Q/M}. Using these connections to simultaneously split $\mathcal{T}_{M}$ and $\mathcal{T}_{\Q}$, the sections 
\begin{align*}
V_A&=\xi Y_A + \mu_c Z_A^c, & \bar{V}_A &= \bar{\xi} \bar{Y}_A + \bar{\mu}_c\bar{Z}_A^c,
\end{align*}
of $(\T_M)^*$ and $(\T_{\Q})^*$ are identified via the relations $\pi^*\bar{\xi} = \xi$ and $\pi^*\bar{\mu}=\mu$.
\end{coro}
\begin{proof}
Writing out the conditions on $V_A$ in terms of its components in a generic scale we have:
\begin{equation}\label{proof: explicit equivalence of tractors}
\begin{aligned}
I^A V_A = 0  \Leftrightarrow  n^a\mu_a + \xi \rho =0, \quad
\nabla_{n} V_A = 0  \Leftrightarrow \begin{cases} \nabla_{n} \xi = n^a\mu_a, \\ \nabla_{n}\mu_a +n^aP_{ab}\xi =0.\end{cases}
  \end{aligned}
\end{equation}	
Working in an adapted scale and making use of Lemma~\ref{Lemma: PropertiesAdaptedScales}, the above equations become equivalent to requiring that each components $\xi$ and $\mu$ are the pullback of corresponding fields $\bar{\xi}$ and $\bar{\mu}$ on $\Q$.
Now, by construction, the isomorphism of Proposition~\ref{Proposition: Isomorphism of tractors on Q/M} maps $D_A\nu=\nu Y_A$ to $D_A \bar{\nu}= \bar{\nu} \bar{Y}_A$. It follows that $\xi Y_A \mapsto \bar{\xi} \bar{Y}_A$ and then by linearity: $V_A -\xi Y_A \mapsto \bar{V}_A - \bar{\xi}\bar{Y}_A$.
\end{proof}
\begin{figure}[H]
\centering
\begin{tikzcd}
0 \arrow[r] & \pi^*T^*\Q(1) \arrow[r,"\pi^*\bar{Z}^a_A"] & \pi^*\T^*_{\Q} \arrow[l,"\pi^*\bar{W}^A_a", dashed, bend right=45,swap]\arrow[r,"\pi^*\bar{X}^A"] & \pi^*\E_{\Q}(1)\arrow[l,dashed, bend right=45,swap,"\pi^*\bar{Y}_A=\bar{s}^{-1}D_A\bar{s}"] \arrow[r]&0 \\
0 \arrow[r] & n^\circ \subset T^*M(1) \arrow[u,equal,"\simeqd"] \arrow[r,"Z^a_A"] &\arrow[l,"\substack{W^A_a\\ I^A=W^A_an^a }",dashed, bend left=45] I^\circ \subset \T^*_{M} \arrow[u,equal,"\simeqd"] \arrow[r,"X^A"] & \E_M(1) \arrow[u,equal,"\simeqd"]\arrow[l,"Y_A=s^{-1}D_A s",dashed, bend left=45] \arrow[r] &0
\end{tikzcd}
\caption{Compatible splittings in adapted scales}\label{Figure: Compatible splittings on Q/M}
\end{figure}

\subsection{The projective structure and the tractor connection of $\Q$}

\subsubsection{The induced projective structure on $\Q$}
We have already established in Proposition~\ref{Connection invariant under flow O} that the projective structure of $M$ induces a projective structure on $\Q_\mathcal{O}$; we now extend this to all of $\Q$.

\begin{prop}\label{Proposition: induced projective structure on Q}
	 Each adapted scale $s$ determines a connection $\bar{\nabla}^s$ on $\Q$. Furthermore, all such connections belong to the same projective structure.
\end{prop}
\begin{proof}
Let $s$ be an adapted scale and let us consider the connection $\nabla$ it determines in the projective class.
As in the proof of Proposition~\ref{Connection invariant under flow O}, we can evaluate the Lie derivative of the connection.
It follows from Lemma~\ref{Lemma: PropertiesAdaptedScales} and Lemma~\ref{ContractionNandWeyl} that:
\[ (\mathcal{L}_{N^s}\nabla^s)^a_{bc} = (N^s)^a P_{bc}, \]
Let $\beta$ be a $1$-form on $\Q$, and $\bar{Z}$ a vector field on $\Q$ to vector fields on $M$. Locally lift $\bar{Z}$ to a vector field $Z$ on $M$ with the property that $\nabla_n Z=0$ and set:
\[\pi^*\bar{\nabla}^s_{\bar{Z}} \alpha = \nabla_{Z}(\pi^*\beta). \]
To check that the definition make sense we must prove that right-hand side is independent of the lift and can be written as the pullback of a form on $\Q$. First note that $\nabla_{N^s} (\pi^*\beta)=0$, hence, two such lifts differ by a vector that is proportional to $N^s$, the formula, if well-defined, is independent of the choice of lift. Second:
\[\mathcal{L}_{N^s} \nabla_{Z}(\pi^*\beta)=\nabla_{[N^s,Z]}(\pi^*\beta) + \pi^*\beta\left(\mathcal{L}_{N^s}\nabla^s(Z, \cdot)\right) +\nabla_{Z} \mathcal{L}_{N^s}\pi^*\beta=0,
 \]
whilst $(\nabla_{\bar{Z}}\pi^*\beta)(N^s)=0$, showing that the connection is well-defined.

If $s_1$ is another adapted scale then \[\nabla^{s_1} = \nabla^{s} - s_1^{-1}\nabla s_1= \nabla^{s}-f^{-1}\nabla f,\] where $f=s_1s^{-1}$. As $D_n s_1= D_n s= 0$ it follows that $f=\pi^* \bar{f}$ and $\frac{df}{f}= \pi^*\frac{d\bar{f}}{\bar{f}}$, so that:
\[\begin{aligned} \nabla_{Z}^{s_1}\pi^*\beta - \nabla_{Z}^{s}\pi^*\beta&= f^{-1}\dd f(\pi^*\beta)(Z) + f^{-1}\dd f(Z)\pi^*\beta \\&= (\pi^*\bar{f})^{-1}\dd (\pi^*\bar{f})\pi^*(\beta(\bar{Z}))+ (\pi^*\bar{f})^{-1}\dd(\pi^*\bar{f})(Z) \pi^*\beta
\\&=\pi^*\left(\bar{f}^{-1}\dd\bar{f}\beta(\bar{Z})\right)+ (\pi^*\bar{f})^{-1}\pi^*(\dd \bar{f})(Z) \pi^*\beta \\&=
\pi^*\left(\bar{f}^{-1}\dd\bar{f}\beta(\bar{Z})+ \bar{f}^{-1} \bar{f}(\bar{Z})\beta\right).
\end{aligned}\]
Therefore all connections constructed in this manner are projectively related.
\end{proof}
The following is an immediate consequence of the definition the connection:
\begin{prop}
If $s=\pi^*\bar{s}$ is an adapted scale then the curvature tensors $R$ and $\bar{R}$ of $\nabla^s$ and $\bar{\nabla}^{s}$ satisfy:
\[\pi^*(\alpha(\bar{R}(\bar{X},\bar{Y})(-))=(\pi^*\alpha)(R(X,Y)(-)),\]
for every $1$-form $\alpha$ and vector fields $\bar{X},\bar{Y}$, on $\Q$ and $X,Y$ are arbitrary lifts of $\bar{X},\bar{Y}$ respectively that satisfy $\nabla_nX = \nabla_n Y=0$.
\end{prop}

\begin{coro}\label{SchoutenTensorMagic}
Let $s=\pi^*\bar{s}$ be an adapted scale then for any vector fields $\bar{Y},\bar{Z}$ on $\Q$ we have:
\[\pi^*(\bar{P}(\bar{Y},\bar{Z}))= P(Y,Z), \]
where $P$ and $\bar{P}$ are the Schouten tensors of $\nabla^s$ and $\bar{\nabla}^{s}$ respectively and $Y$ and $Z$ are lifts of $\bar{Y}, \bar{Z}$ such that $\nabla_n Y=\nabla_n Z=0$.
\end{coro}
\begin{proof}
It follows immediately from the above that:
\[ \pi^*(\bar{R}(\bar{Y},\bar{Z}))=\underbrace{R(Y,Z)}_{\substack{\veq\\(n-1)P(Y,Z)}} - P(Y,Z)=(n-2)P(Y,Z).\]
Since the Schouten tensor $P$ is symmetric ($\nabla^s$ is special) it follows that the Ricci tensor $\bar{R}$ is symmetric (therefore we recover that $\bar{\nabla}^s$ is special), and so $\bar{R}(\bar{Y},\bar{Z})=(n-2)\bar{P}(\bar{Y},\bar{Z})$
\end{proof}

\begin{coro}For any adapted scale $s=\pi^*\bar{s}$:
$\bar{\nabla}^s\bar{s}=0$.
\end{coro}
\begin{proof}
Let $\nu = s^{-(2n+2)}=\pi^*\bar{\nu}$. By definition, we have $\bar{s}^{-(2n+2)}=\bar{\nu}$. These densities are identified with sections of $(\Lambda^nT^*\Q)^{\otimes 2}$, we can apply the definition of $\bar{\nabla}$ and the Leibniz rule to see that:
\[ \pi^*\bar{\nabla}_{\bar{Y}}\bar{\nu}=\nabla_{Y} \pi^*\bar{\nu} = \nabla_{Y}\nu=0. \]
\end{proof}
\subsubsection{The induced tractor connection on $\Q$}

The aim of this section is to show that the projective structures fit together nicely.
\begin{prop}\label{Proposition: InducedConnectionOnT*Q}
The tractor connection on $\T^*_M$ induces a connection on $I^\circ$ and a connection $\bar{\nabla}$ on the bundle $\T^*_{\Q}$.
\end{prop}
\begin{proof}
Since $I^A$ is parallel if $I^AV_A=0$ then $I^A\nabla_a V_A=0$, so $\nabla$ restricts to a connection on $I^\circ$. To define a connection on $\mathcal{T}^*_{\Q}$, let $\bar{Z}$ be a vector field on $Q$ and let $Z$ be any local lift of $\bar{Z}$ to $M$ that satisfies $\nabla_n Z=0$.
Let $\phi: \pi^* \T_{\Q} \rightarrow I^\circ$ denote the bundle morphism established in the previous section and set:
\[ \pi^*\bar{\nabla}_{\bar{Z}} V= \phi^{-1}\nabla_{Z}\phi(\pi^* V).  \]
This is independent on the choice of lift $Z$ as any other lift $Z'$ is such that $(Z-Z')\propto n$ and $\nabla_n \phi(\pi^* V)=0$. This second identity can be seen using an adapted scale in which the components $\xi$ and $\mu_c$ of $\phi(\pi^*V)$ must satisfy $\nabla_n \xi =0$, $\nabla_n \mu_c=0$ and $\mu_cn^c=0$. Since in adapted scales $P_{ab}n^a=0$, this means precisely that the tractor is covariantly constant in the direction of $n$.
It remains to check that $\nabla_n \nabla_Z \phi(\pi^*V) =0$ so that the right-hand side is indeed a pullback section.

For any choice of adapted scale $\nu$, setting $N=\nu n$, we have the identity for cotractors $V_A$:
	\begin{align}
	\nabla_N (\nabla_Z V_A) - \nabla_Z (\nabla_N V_A) &= [Z,N]^a \nabla_a V_A -N^a Z^b\Omega_{ab}{}^B{}_A V_B
	\end{align}
	Using this with $V=\phi(\pi^*V)$ and our assumption on the lift: 
	\begin{align}
	\nabla_N (\nabla_Z V_A) &= -N^a Z^b\Omega_{ab}{}^B{}_A V_B.
	\end{align}
	The vanishing of this last term, concluding the proof, will be given by the following Lemma:
	\begin{lemm}\label{Lemma: n^aF_a =0}
		\begin{equation}
		n^a \Omega_{ab}{}^A{}_B =0.
		\end{equation}
	\end{lemm}
	\begin{proof}
		The tractor curvature is a projectively invariant tensor which, in any given scale, is, by Eq.~\eqref{TractorCurvature}, completely parametrised by the (projective) Weyl and (projective) Cotton tensors. Evaluating their contraction with $n^a$ in the scale $\sigma$ one sees from \eqref{Riemann in the scale sigma}, \eqref{Levi-civita in the scale sigma} that their contraction with $n^a$ always vanishes. This extends to the points where $\sigma=0$ by continuity.
\end{proof}
It follows that $\nabla_N(\nabla_Z V_A)$ vanishes and $\bar{\nabla}$ is well-defined.
\end{proof}

\begin{prop} \label{Proposition: TractorGeometryofQ}\phantom{a}
\begin{enumerate}
\item $H^{AB}$ induces a parallel section $\bar{H}^{AB}$ on $\Q$. It is a symmetric pairing of rank $n-1$ on $\T_{\Q}^*$.
\item $I_A=D_A\sigma$ induces a parallel section $\bar{I}_A$ of $\T_{\Q}^*$ and $\bar{I}_A$ is in the radical of $\bar{H}^{AB}$.
\item $\bar{I}_A=D_A\bar{\sigma}$.
\end{enumerate}
\end{prop}
\begin{proof}
\begin{enumerate}
\item Let $\tilde{H}^{AB}$ denote the induced metric in $I^\circ$, let $\bar{V}_A,\bar{W}_B$ be sections of $\T_{\Q}^*$ that we can pullback an identify with sections  $V_A,W_B$ of $I^\circ$ that satisfy, $\nabla_n V_A=\nabla_n W_B= 0.$ Since $H^{AB}$ is parallel it follows that $H^{AB}V_AW_B$ is constant along the fibres of $M \rightarrow \Q$; this can this be used as the definition of $\bar{H}^{AB}$. 
To show that $\bar{H}^{AB}$ is parallel observe that we have the scalar identity:
\[ \bar{\nabla}_a(\bar{H}^{AB}\bar{V}_A\bar{W}_B)=\bar{\nabla}_a(\bar{H}^{AB}\bar{V}_A\bar{W}_B) - \bar{H}^{AB}\bar{\nabla}_a\bar{V}_A\bar{W}_B - \bar{H}^{AB}\bar{V}_A\bar{\nabla}_a\bar{W}_B, \]
which, by definition of $\bar{\nabla}$ and $\bar{H}^{AB}$, must pull back to the identity:
\[ \nabla_a(H^{AB}V_AW_B)=\nabla_aH^{AB}V_AW_B - H^{AB}\nabla_aV_AW_B - H^{AB}V_A\nabla_aW_B=0. \]
In the above we have defined\footnote{We omit, by abuse of notation, the bundle isomorphism}: $V_A=\pi^*\bar{V}_A$ and $W_A=\pi^*\bar{W}_B$.
\item This follows from the fact that $\nabla_c D_A \sigma =0$ and $I^AD_A \sigma = \Phi_{AB}I^AI^B=0$. It is immediate from the definition of $\bar{\nabla}$ that $\bar{I}_A$ is parallel for $\bar{\nabla}$.
\end{enumerate}
\end{proof}

\begin{prop}\label{Proposition: normal tractor connection on Q}
The connection $\bar{\nabla}$ on $\T^*_{\Q}$ is the Cartan normal connection associated to the projective class of $[\bar{\nabla}^s]$ where $s$ is any adapted scale. 
\end{prop}

\begin{proof}
Choose a fixed adapted scale $s$ and introduce the splittings of $\T_M^*$ and $\T_Q^*$ as in Corollary~\ref{Proposition: explicit tractor isomorphism}.
Let $\bar{V}_A= \bar{\xi}\bar{Y}_A + \bar{\mu}_c\bar{Z}^c_A$ be a section of $\T_Q^*$, we have from the aforementioned Corollary that: \[V_A=\pi^*\bar{V}_A= \xi Y_A + \mu_c Z^c_A\] where $\xi = \pi^*\xi$ and $\mu=\pi^*\mu$. 
Let $\bar{Z}$ be a vector field on $\Q$ and $Z$ a lift to $M$ such that $\nabla_n Z=0$, then:
\[ \nabla_Z V_A= (\nabla_Z\xi -\mu(Z) )Y_A + (\nabla_Z \mu_c +P(Z,-)_c\xi)Z^c_C.\]
By definition of the connection $\bar{\nabla}$, the Leibniz rule and Corollary~\ref{SchoutenTensorMagic} we have:
\[ \begin{cases} \nabla_Z \xi - \mu(Z)= \pi^*\left(\bar{\nabla}_{\bar{Z}}\bar{\xi}-\bar{\mu}(\bar{Z})\right)\\(\nabla_Z \mu +P(Z,-)\xi)=
\pi^*\left(\nabla_{\bar{Z}}\bar{\mu} + \bar{P}(\bar{Z},-)\xi \right). \end{cases} \]
Hence, using again Corollary~\ref{Proposition: explicit tractor isomorphism}, it follows that:
\[\bar{\nabla}_c\bar{V}_A=(\bar{\nabla}_c \bar{\xi} - \bar{\mu}_c)\bar{Y}_A + (\bar{\nabla}_c\bar{\mu}_a+\bar{P}_{ca})Z^c_A.\]
\end{proof}

\subsubsection{Projective compactification of order 1}

It now follows from Proposition~\ref{Proposition: TractorGeometryofQ} and Proposition~\ref{Proposition: normal tractor connection on Q} that $\Q$ is equipped with a normal solution to the metrisability equation of rank $n-1$ whose geometry is therefore described by~\cite[Theorem 3.14]{Flood:2018aa}:

\begin{theo} \label{Theorem: Projective compactification order 1}
Let $g^{ab}=\bar{\sigma}^{-2}\bar{H}^{AB}\bar{Z}_A^b\bar{Z}_B^b$ on $\Q\cap \{\bar{\sigma}\neq 0\}$.
Then the restriction of $g_{ab}$ to $\Q\cap\{\pm \overline{\sigma} >0\}$ is projectively compact (Ricci-flat) Einstein metric of order $1$ and $\Q\cap\{\pm \overline{\sigma} \geq0\}$ is a projective compactification with projective boundary $\mathcal{H}=\Q\cap\{\bar{\sigma}=0\}$.
\end{theo}

\subsection{An additional induced tractor bundle on $\Q$}\label{AdditionalTractorBundle}
The isomorphisms in Proposition~\ref{Proposition: Isomorphism of tractors on Q/M} are convenient for the classical description of the projective geometry on $\Q$. However, the infinitesimal geometric structure on $\Q$ is \emph{richer} than that of the usual projective tractor geometry.
We will first illustrate this by constructing from the ambient projective tractor bundle $\T_M \to M$ a vector bundle $\tilde{\T} \to \Q$ on $\Q$,  whose fibres of dimension $n+2$ will be able to carry additional information compared to $\T_{\Q} \to \Q$ (which is a rank $n+1$ vector bundle). This might at first seem incidental, but it turns out to be important in elucidating the geometry at the projective boundary $\mathcal{H} \subset \Q$. 

The bundle $\widetilde{\T}$ inherits a connection and parallel sections $\tilde{I}^A$, $\tilde{H}^{AB}$, mimicking the original setup on $M$ seen here, however, as a structure on $\Q$. A natural question one might ask is: does this vector bundle correspond to the tractor bundle of a Cartan geometry on $\Q$ and what, if any, additional geometric structures may it induce?

The answer we propose is that this is in fact an alternative form of \emph{projective} geometry on $\Q$, however, based on a \emph{non-effective} homogeneous model. 

Referring to the model of our holonomy reduction discussed in the Introduction and Section~\ref{model}, one might observe there that the base $\R P^n$ arises equivalently as either the set of planes containing the line generated by $I$ or, the set of projective lines passing through the fixed point $[I]$. It can then be described as the Klein geometry $\tilde{G}/\tilde{P}$ where $\tilde{G}$ is the isotropy subgroup of $[I]$ in $PGL(n+1)$ and $\tilde{P}$ the stabiliser of some given plane containing $\R I$. It is clear that the action has a relatively large kernel conjugate to: 
\[ \tilde{K}=\left\{ \begin{pmatrix} c  & \upsilon \\ 0 & \lambda I_{n+1} \end{pmatrix} \!\!\!\mod Z(GL(n+2))\right\}.\]

Given any Cartan geometry\footnote{$C$ is $\tilde{P}$-principal bundle.} $(C \to \Q, \omega)$ with this model, it is natural to construct a bundle $\mathscr{T}$ associated to the restriction to $\tilde{P}$ of the $(n+2)$-dimensional representation $u \mapsto \lvert\det u\rvert^{-\frac{1}{n+2}}u$ of $PGL(n+1,\R)$. There is also a canonical line bundle $\mathscr{L} \rightarrow \Q$ associated to the 1-dimensional representation of $\tilde{P}$ given by:
\[ \begin{pmatrix} c& \upsilon_\nu & \upsilon_0 \\ 0 & A^\mu_\nu & 0 \\ 0 & \Upsilon_\mu & a \end{pmatrix} \!\!\!\mod Z(GL(n+2)) \mapsto \left\lvert\frac{c}{a}\right\rvert^\frac{1}{n+2}. \]
As with projective densities, let $\mathscr{L}(w)$ denote the $w$-th power of $\mathscr{L}$. For simplicity, in this brief discussion, we shall assume that the line bundle corresponding to the representation $G\in \tilde{P} \mapsto \frac{c}{a}$ is oriented and can therefore be identified with a power of the first.
A generic $\tilde{\mathfrak{g}}$-valued 1-form on $\Q$
can be parametrised as follows:
\begin{equation}\label{GeneralParametrisationConnection} \begin{pmatrix}-\frac{1}{n+2}\alpha^\sigma_\sigma + \frac{n+1}{n+2}\kappa & \upsilon_\nu& \upsilon_0\\ 0 & \alpha^\mu_\nu - \frac{1}{n+2}\alpha^\sigma_\sigma \delta^\mu_\nu -\frac{1}{n+2}\kappa\delta^\mu_\nu & \omega^\mu \\ 0 & P_\nu & -\frac{1}{n+2}\alpha^\sigma_\sigma -\frac{1}{n+2}\kappa \end{pmatrix}. \end{equation}
The 1-form $\kappa$ parametrises a linear connection on $\mathscr{L}$ with local connection form $\frac{1}{n+2}\kappa$. Modulo the kernel $\mathfrak{k}$, this leads to the usual parametrisation of a projective connection and a natural condition on $\omega$ is to require that $\tilde{\omega}=\omega \!\!\!\mod \mathfrak{k}$ is the normal Cartan connection of the corresponding projective class. We make this assumption in the rest of the paragraph, thus $P_\nu$ is completely determined in terms of the Projective Schouten tensor.

Let us now make a few further general observations about these Cartan geometries. First, we see that there is always a distinguished direction in $\mathscr{T}$; this should be identified with the origin $[I]$ in the definition of the non-effective homogenous model. Secondly, we have a natural exact sequence realising $\mathscr{T}$ as an extension of a weighted version of the usual tractor bundle $\T_Q$:
\begin{center}
\begin{tikzcd} 0\arrow[r] & \E_Q(-\frac{n+1}{n+2})\otimes \mathscr{L}(n+1) \arrow[r] & \mathscr{T} \arrow[r] & \T_Q(\frac{1}{n+2})\otimes \mathscr{L}(-1) \arrow[r] &0.  \end{tikzcd}
\end{center}

Now, if we assume that there is a distinguished \emph{parallel} section $\tilde{I}$ of $\mathscr{T}$ in the privileged direction  there are local gauges in which $\kappa= \frac{1}{n+1}\alpha^\mu_\mu$ and it is apparent that we should identify $\mathscr{L}$ with $\mathcal{E}(\frac{1}{n+2})$. The short exact sequence then becomes:
\begin{figure}[h!]
\begin{center}
\begin{tikzcd} 0\arrow[r] & C^{\infty}(\Q) \arrow[r, "\tilde{I}"] & \mathscr{T} \arrow[r, "\Pi"] & \T_{\Q} \arrow[r] &0  \end{tikzcd}
\end{center}
\caption{Decomposition of $\mathscr{T}$\label{Decomposition_Tscr}\protect\footnotemark}
\end{figure}
\footnotetext{We will abusively use the notation $C^\infty(\Q)$ to denote both the module of smooth real-valued functions on $\Q$ and the (trivial) bundle of which they are sections. }

Under the above compatibility assumption we therefore see that parallel sections of $\mathscr{T}$ canonically project to parallel sections of $\T_\Q$; so a holonomy reduction of the Cartan geometry $(C,\omega)$ given by parallel sections of $\mathscr{T}$ leads to a holonomy reduction of the underlying normal projective geometry. For instance, if we assume that there is a non-degenerate parallel section $\tilde{H}^{\tilde{A}\tilde{B}}$ of signature $(2,n)$ such that $\tilde{\Phi}_{\tilde{A}\tilde{B}}\tilde{I}^{\tilde{A}}\tilde{I}^{\tilde{B}}=0$, where $\tilde{\Phi}$ is the inverse of $\tilde{H}$, then we get a parallel section: $H^{AB}=\Pi_{\tilde{A}}^A\Pi_{\tilde{B}}^B\tilde{H}^{\tilde{A}\tilde{B}}$. Additionally, the section: $\Pi^B_{\tilde{B}}\tilde{H}^{\tilde{A}\tilde{B}}$ induces an isomorphism $\mathcal{T}_\Q^* \simeq \tilde{I}^\perp$. Since $\tilde{I} \in \tilde{I}^\perp$, it induces a cotractor $I_A$ such that $H^{AB}I_A=0$, as expected.

There are few unresolved issues in pushing the general discussion further. For instance, the general connection in Eq.~\eqref{GeneralParametrisationConnection} has unspecified components $\upsilon_0,\upsilon_\nu$ that do not seem immediately linked to the underlying normal projective structure. However, one may expect that the existence of a parallel section $H^{\tilde{A}\tilde{B}}$ could exhaust these remaining degrees of freedom. Indeed, an argument in favour of this is that there will be an induced Ricci-flat Einstein metric $g$ on $\Q$ and, where $g$ is non-degenerate, we will be able to find local gauges of the form:
\[\begin{pmatrix}0& -\omega^\mu\eta_{\mu\nu}& 0\\ 0 & \alpha^\mu_\nu & \omega^\mu \\ 0 & 0 &0\end{pmatrix},\] 
where $\alpha$ is the Levi-Civita connection form expressed in an orthogonal basis and $\eta=\textrm{diag}(-1,1,\dots,1)$. 

Another related issue is that it is not clear how to split the exact sequence in Figure~\ref{Decomposition_Tscr}. From the strict of point of view of the bundle $\mathscr{T}$, this is a choice of a point in the affine space $\mathfrak{a}=\{ T_{\tilde{A}} \in \mathscr{T}^* \, \lvert \, T_{\tilde{A}}\tilde{I}^{\tilde{A}}=1\}$, directed by sections of $\tilde{I}^{\circ} \simeq\T_Q^*$. However, it is unclear how one may interpret this in terms of a geometric structure on $\Q$. 
Supposing a choice can be made, if we denote $\rotpi$ the inverse induced splitting section, then we note that it follows directly from Eq.~\eqref{GeneralParametrisationConnection} that the connection will act on $T^{\tilde{A}}= f \tilde{I}^{\tilde{A}}+ \rotpi^{\tilde{A}}_A t^A$ ($f\in C^{\infty}(\Q), t\in \Gamma(\T_{\Q})$) according to:
\begin{equation}\label{Desiredformofconnection} \nabla_c T^{\tilde{A}}= (\nabla_c f + \upsilon_{cB} t^B)\tilde{I}^{\tilde{A}} + \rotpi^{\tilde{A}}_A \nabla_c t^A, \end{equation}
where $\upsilon_{cA}$ is a co-tractor valued $1$-form determined by the connection and depending on the splitting.

In this section we construct a vector bundle $\widetilde{\T}\rightarrow \Q$ that we conjecture carries the structure of $\mathscr{T}$. In the case we study, the distinguished section $\tilde{I}$ is parallel and we will show that $\widetilde{\T}$ has the same decomposition structure as in Figure~\ref{Decomposition_Tscr}.

The bundle $\widetilde{\T}\rightarrow \Q$ is constructed as follows: if $\nu$ is a density on $\widetilde{M}$ that restricts to a nowhere vanishing density on $M$, the flow of $N^a =\nu n^a$ defines a canonical action of $\R$ on the tractor bundle $\pi_{\T}: \mathcal{T}_{M} \rightarrow M$ that parallel transports tractors along the integral curve of $N$, $t\mapsto \phi^N_t(x)$.
The quotient $\pi_{\T_{M}\to \tilde{\T}}: \mathcal{T}_{M} \longrightarrow \widetilde{\mathcal{T}}$ is a rank $(n+2)$ vector bundle $\widetilde{\T} \to \Q$, with projection defined by factorisation of the map $\pi_{\tilde{\T}}$ (see Figure~\ref{DefinitionTauTilde}); this construction is independent of the choice of scale $\nu$. This leads to:
\begin{defi}
Define $\tilde{\T}$ to be the quotient space $\T/\R$ constructed as above.  It naturally fits into the commutative square:
\begin{figure}[H]
\centering
\begin{tikzcd}
\T_{M} \arrow[r, "\pi_{\T_M\to \tilde{\T}}" ] \arrow[d,"\pi_{\T}"] & \tilde{\T} \arrow[d, "\pi_{\tilde{\T}}"]\\
M \arrow[r,"\pi"] & \Q
\end{tikzcd}
\caption{\label{DefinitionTauTilde}Defining diagram of the bundle $\tilde{\T} \rightarrow \Q$.}
\end{figure}
\end{defi}

A point in the fibre $\pi_{\tilde{\T}}^{-1}(\{q\})$ can be thought of as a parallel vector field along the integral curve $q\in \Q$ of $N$ or, in other words, a covariantly constant section along $\pi^{-1}(\{q\})$. From Figure~\ref{DefinitionTauTilde}, it can be seen that sections of $\mathcal{T}_M$ which are covariantly constant along the fibres of $M \to \Q$ map to sections of $\tilde{\T}$. 

Conversely, a section $\tilde{V}$ of $\tilde{\T}$ gives rise to such a section of $\T_M$. Indeed, we can define a map on $M=\coprod_{q\in \Q} \pi^{-1}(\{q\})$, that to each $x\in \pi^{-1}(\{q\})$ gives the value of the section $V(q)$ at $x$. Using the maps $q,t$ of a local trivialisation $\psi: \pi^{-1}(U_i) \rightarrow U_i\times \R$, locally this coincides with $V(x)=\bar{V}(q(x))(t(x))$ where for each $q\in U_i\subset \Q$, $\bar{V}(q)$ is the solution of the ordinary differential equation,
\[\begin{cases} \nabla_{\dot{\gamma}} \bar{V}(t)=0, \\ \bar{V}(0)=\tilde{V}(q)(\psi^{-1}(q,0)), \\ \gamma(t)=\psi^{-1}(q,t)=\phi^{\tilde{N}}_t(x_0), \textrm{where $x_0=\phi^{-1}(q,0)$}. \end{cases}\]
Above, $\tilde{N}$ is the image of the vector field $\partial_t$ on $U_i \times \R_t$ under the diffeomorphism $\psi$; the section is smooth by smoothness with respect to initial conditions.
We record this observation in the following lemma.
\begin{lemm}\label{correspondence_sections_T_tilde}
Sections of $\widetilde{\T}$ are in one-to-one correspondence with sections of $\T$ that are covariantly constant along the fibres of $M\longrightarrow \Q$.
\end{lemm}
As claimed we see that the bundle comes equipped with a natural set of geometric data.
\begin{prop}
 $\widetilde{\T}$ inherits from $\T$ the follow data:
\begin{enumerate}
\item\label{connection_tilde} a natural linear connection defined by the relation:
\[ \widetilde{\nabla_{\bar{X}}T}_p= (\nabla_X \widetilde{T})_p, \]
where $p\in M$, $\bar{X}$ is a vector field on $Q$ near $q=\pi(p)$, $T\in \Gamma(\widetilde{\T})$, $X$ is lift of $\bar{X}$ near $p$ that satisfies $\nabla_n X=0$, and $~\widetilde{}$ denotes the correspondence of Lemma~\ref{correspondence_sections_T_tilde},
\item canonical sections $\tilde{H}^{AB}$ and $\tilde{I}^A$ of $\tilde{\mathcal{T}}\otimes\tilde{\mathcal{T}}$ and  $\tilde{\mathcal{T}}$.
\end{enumerate}
Furthermore, the sections $\tilde{I}$ and $\tilde{H}$ are parallel for the connection defined in~\ref{connection_tilde}.
\end{prop}
\begin{proof} 

\begin{enumerate}\item The construction is similar to that in the proof of Proposition~\ref{Proposition: InducedConnectionOnT*Q}.  \item Since $\nabla_a H^{AB}=0$ and $\nabla_a I^A=0$ these both descend to corresponding sections $\tilde{H}^{AB}$ and $\tilde{I}^A$.\end{enumerate}  That they are parallel follows immediately from the definition of the connection.\end{proof}

To develop further the idea that this data should be thought of as an intrinsic Cartan geometry on $\Q$ modeled on $\tilde{G}/\tilde{P}$, we now confirm that we have the expected decomposition sequence of Figure~\ref{Decomposition_Tscr} realising it as an extension of the standard projective tractor bundle.
\begin{prop}\label{Proposition: tractor isomorphism} There are canonical isomorphisms: 
\begin{align}
	\tilde{I}^{\circ} &\simeq (\T_{\Q})^*, & \tilde{\mathcal{T}}/\tilde{I} &\simeq \T_{\Q}.
\end{align}
In particular, we have the decomposition sequence:
\begin{center}
\begin{tikzcd}
0 \arrow[r] & \T_{\Q}^* \simeq \tilde{I}^{\circ} \arrow[r] &\widetilde{\T}^* \arrow[r, "\tilde{I}"]& C^{\infty}(\Q) \arrow[r] & 0
\end{tikzcd}
\end{center}
\end{prop}
\begin{proof}
We refer the reader to the proof of Proposition~\ref{Proposition: Isomorphism of tractors on Q/M} as the arguments are similar.
\end{proof}
As in the discussion of the model, one can split the exact sequence with a choice of section of $\widetilde{\T}^*$ that does not annihilate the distinguished tractor $\tilde{I}$; in fact on $\Q_{\mathcal{O}}$, we have a canonical choice:
\begin{lemm}
The cotractor $L_C=\frac{1}{2}D_C(\sigma^{-1}\lambda)$ descends to a section $\tilde{L}_{\tilde{C}}$ of $\widetilde{\T}^*$ such that $\tilde{L}_{\tilde{C}}\tilde{I}^{\tilde{C}}=1$.
\end{lemm}
\begin{proof}
Using Lemma~\ref{Lemma: Relations ISigma, XLambda, etc}:
\[\begin{aligned} n^d\nabla_d L_C=\frac{1}{2}I^DD_DD_C(\sigma^{-1}\lambda)
&=\frac{1}{2}I^D\left(2\sigma^{-3}\lambda D_DD_C\sigma -2\sigma^{-2}D_{(D}\lambda D_{C)} \sigma +\sigma^{-1}D_DD_C\lambda  \right) \\
&= - \sigma^{-2}\frac{1}{2}I^DD_D\lambda D_C\sigma + \sigma^{-1}I^D\Phi_{DC}\\
&=  -\sigma^{-1}D_C\sigma + \sigma^{-1}D_C\sigma =0.
\end{aligned}\]
The first claim then follows from Lemma~\ref{correspondence_sections_T_tilde} and a straightforward computation, using again Lemma~\ref{Lemma: Relations ISigma, XLambda, etc}, shows the second.
\end{proof}

This provides us with the means to test the following statements on $\Q_\mathcal{O}$: 
\begin{coro} \phantom{p}\begin{enumerate}\item The bundle $\widetilde{\T}$ with its connection $\tilde{\nabla}$ and parallel tractor $\tilde{I}$, endows $\Q$ with a Cartan geometry modeled on the non-homogenous model of projective geometry $\tilde{G}/\tilde{P}$. \label{conj1}
 \item The linear connection on $\T_{\Q}^*$ induced by the isomorphism $\T_{\Q}^* \simeq \tilde{I}^{\circ}$ coincides with the normal projective connection $\bar{\nabla}$.
\end{enumerate}
\end{coro}
\begin{proof}
We only sketch the proof of~\ref{conj1} on $\Q_{\mathcal{O}}$. Since $\tilde{I}$ is parallel it follows from its definition that the connection $\tilde{\nabla}$ will adopt the desired form given by Eq.~\eqref{Desiredformofconnection}, where the connection on $\mathcal{\T}_Q$ is induced from $\tilde{\nabla}$ by the isomorphism of Proposition~\ref{Proposition: tractor isomorphism}. In particular, we may compute the connection cotractor 1-form $\upsilon$ from the $\tilde{\nabla}_b\tilde{L}_{\tilde{C}}$.
Working in the scale determined by $\sigma$ on $\mathcal{O}$, it can be determined from:
$\frac{1}{2}W^B_bD_BD_C(\sigma^{-1}\lambda)=\frac{1}{2}\sigma^{-1}W^B_b\nabla_a\nabla_b \lambda Z^a_A=\sigma g_{ab}Z^b_B$, in the notation of Section~\ref{section: Geometry of the open orbit O}.
Pushing down to $\mathcal{Q}_{\mathcal{O}}$ through the different identifications of this and the previous sections, we arrive at:
\[\upsilon_{aB}= -\bar{\sigma} h_{ab}\bar{Z}^b_B,\]
where $h$ is the metric of Proposition~\ref{Proposition: invertible metric on Q_O}.
\end{proof}
In this situation, where the geometry arises as an inherited structure from an ambient space, there is an obvious interpretation of the additional degrees of freedom in the non-effective Cartan geometry as the remnants of the projective degree of freedom of the projective geometry on the ambient space; however it is still unclear how they should be interpreted in terms of $\Q$ alone.

It should also be noted that Proposition~\ref{Proposition: tractor isomorphism} does not rely on the existence of the parallel tractor $\tilde{H}$ on $\widetilde{\T}$. Taking it into account adds a further interesting element to the geometric picture surrounding of $\widetilde{\T}$; relating it to conformal geometry. By Proposition \ref{Proposition: invertible metric on Q_O}, $\Q_{\mathcal{O}} := \Q\cap \{\sigma \neq0 \}$ is equipped with a conformal class of metrics $[h^{ab}]$ and therefore defines a \emph{conformal} tractor bundle.
\begin{prop}\label{Proposition: ambient tractors extend conformal} On $\Q_{\mathcal{O}}$ the bundle $\tilde{\mathcal{T}}$ is canonically isomorphic to the conformal tractor bundle. What is more the induced connection $\tilde{\nabla}$ coincides with the normal conformal tractor connection.
\end{prop}
\begin{proof}
By Theorem \ref{TheoFibreBundleInside BIS}, the line bundle $\mathcal{O} \to \Q_{\mathcal{O}}$ has a preferred section $s : \Q_{\mathcal{O}} \to \mathcal{O}\cap\{\lambda=0\}$. This map thus canonically identifies $\tilde{\mathcal{T}}$ with the pull back of the projective tractor bundle $i^*\mathcal{T}_{\mathcal{O}}\to \mathcal{O}\cap\{\lambda=0\}$ as well as the corresponding induced connections.
Now we recall that the submanifold $\mathcal{O}\cap\{\lambda=0\}$ can also be thought as part of the conformal boundary for a projective compactification of order $2$.  What is more the corresponding conformal structure coincides with the conformal structure on $\Q_{\mathcal{O}}$ (this follows from their respective definitions). However, by the results from \cite[Section 4]{cap_projective_2016-1}, the pullback bundle $i^*\mathcal{T}_{\mathcal{O}}$  of projective tractors along the projective boundary is canonically identified with the conformal tractor bundle and the restriction of the projective tractor connection coincides with the normal conformal tractor connection.
\end{proof}
In other terms $\tilde{\T}$ (and its induced connection $\tilde{\nabla}$) are continuous extensions of the conformal tractor bundle (and the conformal tractor connection) from $\Q_{\mathcal{O}}$ to the whole of the projective compactification $\Q$.

The points we develop in this section are the first steps towards the desirable picture in which this structure can be constructed from geometrical data on $\Q$ alone; they suggest an interesting picture and interplay between projective and conformal geometry that we hope to unravel in future work~\cite{BorthwickHerfray2}.

\section{The geometry of $\Sigma \to \mathcal{H}$}\label{section: The geometry of Sigma to H}

\subsection{The geometry of $\Sigma$}\label{ssection: the geometry of Sigma}

The hypersurface $\Sigma \subset M$ is the projective boundary of $\mathcal{O}$ given by a holonomy reduction from $\text{PGL}(n+1)$ to $\text{Aff}(n+1)$: this follows from $\nabla_c I_A=0$ and $I_{A} = D_A\sigma$ (cf Lemma \ref{Lemma: Relations ISigma, XLambda, etc}).

We have the following short sequences, dual to that in~\cite[Corollary 3.15]{Flood:2018aa},
\begin{prop}\mbox{}\label{Proposition: short sequences of tractors on Sigma}
\begin{center}
	\begin{tikzcd}
		&0\arrow[d] & 0\arrow[d]\\&\mathcal{E}_\Sigma(-1)\arrow[r,"\simeq"] \arrow[d,hook,"X"]& i^*\E(-1)\equiv\E(-1)|_\Sigma \arrow[d,hook,"X"] \\ 0 \arrow[r]& \T_\Sigma  \simeq i^* (I_A)^{\circ} \arrow[r,hook] \arrow[d,two heads, "Z"] & i^*\T_{M} \arrow[r,two heads,"D_A\sigma"] \arrow[d,two heads, "Z"]& C^\infty(\Sigma)\arrow[d,"\simeq"] \arrow[r]& 0
		\\0\arrow[r]& T\Sigma(-1)\arrow[d]\arrow[r,hook]& i^*TM(-1)\arrow[d] \arrow[r,"\nabla_a \sigma", two heads]& C^\infty(\Sigma)\arrow[r] & 0\\&0&0
	\end{tikzcd}
\end{center}
\end{prop}
\begin{proof}
Since $\sigma$ is boundary defining function for $\Sigma$ we have $J^1\mathcal{E}_{\Sigma}(1) = i^*(J^1\mathcal{E}_{M}(1))/ J^1\sigma$ or equivalently $(\T_{\Sigma})^* = i^*(\T_{M})^* / I_A$. Passing to the dual gives $\T_{\Sigma} = i^*\big(I_A\big)^{\circ} \subset i^*(\T_{M})$.
\end{proof}

Recall that we denote by $\Phi_{AB}$ the inverse of $H^{AB}$. Since we have $\nabla_c \Phi_{AB}=0$, $\nabla_c I^A=0$  and $I^2=0$ on $M$ it follows from the isomorphisms in Proposition~\ref{Proposition: short sequences of tractors on Sigma} that we have
\begin{prop}\label{Proposition: Induced tractor geometry on Sigma}
The tractor bundle $\T_{\Sigma}$ of $\Sigma\subset M$ inherits from $M$ a projective tractor connection $\nabla_c$, a covariantly constant section $I^A$ and a covariantly constant pairing $\overline{\Phi}_{AB}$ with one dimensional kernel such that $\overline{\Phi}_{AB} I^B=0$.
\end{prop}

\begin{rema}
We point out that contrary to what we observed in Proposition~\ref{Proposition: normal tractor connection on Q} for the quotient $\Q$,  the projective tractor connection $\nabla_c$ induced on the boundary, is generically \emph{not} normal (as a projective connection).
\end{rema}
The geometry induced on $\Sigma$ is far richer than just a projective geometry. Let us here restrict to $\Sigma^{\pm} \subset M^{\pm}$ where $\lambda$ is nowhere vanishing  and thus defines a preferred scale\footnote{The Cartan geometry induced on $M^0$ is similar to that of null infinity, the conformal boundary of an asymptotically flat spacetime. This has been studied in details in \cite{herfray_asymptotic_2020,herfray_tractor_2022} and can be thought as a conformal version of the geometries induced on $M^{\pm}$, see \cite{Herfray:2021qmp}.}. In this case, the induced geometries are known as Carrollian geometries \cite{duval_carroll_2014}:
\begin{theo}\label{Theorem: induced geometry on Sigma}
	The hypersurface $\Sigma^{\pm} \subset M$ is naturally equipped with
	\begin{itemize}
		\item a vector field $N^a$
		\item a degenerate metric $q_{ab}$ of rank $n-1$ satisfying $q_{ab}N^b=0$, $\mathcal{L}_{N}q_{ab}=0$,
		\item a compatible affine connection $\nabla_c$ satisfying $\nabla_cN^a=0$, and $\nabla_c q_{ab}=0$.
	\end{itemize}
On $\Sigma^{\pm}$, the signature of $q_{ab}$ is $(0, \mp, + ... +)$.
\end{theo}
\begin{proof}
 Recall that, from Lemma \ref{Lemma: Relations ISigma, XLambda, etc}, $\Phi_{AB} = D_A D_B \lambda$ at every point of $M$. On $M^{\pm}$ it follows from \cite{cap_holonomy_2014,Flood:2018aa} that, in the scale $\lambda$, the representative connection $\nabla_c$ in the projective class is the Levi-Civita connection of the metric $g_{ab}$ defined by
\begin{equation}\label{Proof: tractor metric on Sigma}
\Phi_{AB} = \lambda( Y_A Y_B + g_{ab} Z^a_A Z^b_B).
\end{equation}
What is more when $\Phi_{AB}X^AX^B = \pm|\lambda|\neq0$, since $\Phi_{AB}$ has signature $(--+...+)$, $g_{ab}$ has signature $(-\mp+...+)$.  
Now Lemma \ref{CovariantDerivativesDirectionn} implies, when evaluated in the scale $\lambda$ that $\rho=0$ on $\Sigma^{\pm}$ i.e. $i^*(I^A) = n^a W_a^A$. Since by Lemma \ref{Lemma: Relations ISigma, XLambda, etc} $i^*(D_A\sigma) = i^*(\Phi_{AB}I^B)$ we find, in the scale $\lambda$,
\begin{align*}
\sqrt{\lambda}\nabla_a\sigma = g_{ab} N^b
\end{align*}
where
\begin{equation}
N^a := i^*(\sqrt{\lambda} n^a)
\end{equation}
 is therefore the normal to $\Sigma$. Since $0= I^2= g_{ab} N^a N^b$ this normal is null and the pull back 
\begin{equation}
q_{ab} := i^*g_{ab}
\end{equation}
is degenerate with one dimensional kernel spanned by $N^a$ and signature $(0\mp+...+)$.

We already know that $\nabla_a$ is the Levi-Civita connection of $g_{ab}$. To conclude that it induces a connection on the tangent bundle of $\Sigma$ which satisfies,
\begin{align*}
\nabla_c q_{ab}&=0, & \nabla_c N^a&=0.
\end{align*}
we need only to verify that $i^*(\nabla_a N^b)=0$. However this follows from evaluating  \eqref{IParallel} in the scale $\lambda$. One concludes by noting that $\mathcal{L}_N q_{ab} = i^*(\mathcal{L}_N g_{ab}) = 2i^*(\nabla_{(a} N_{b)})=0$.
\end{proof}
Carrollian geometries have been extensively studied in e.g. \cite{ duval_carroll_2014,bergshoeff_carroll_2017,bekaert_connections_2018,morand_embedding_2020,Figueroa-OFarrill2022,figueroa-ofarrill_non-lorentzian_2022}.
As was shown in \cite{hartong_gauging_2015,Herfray:2021qmp} they define a unique normal Cartan geometry modelled on the homogeneous spaces $\mathrm{Ti}^n$ and $\mathrm{Spi}^n$ of \eqref{Boundary Homogeneous space}.

\begin{prop}(\cite{Herfray:2021qmp})\label{Proposition: Carollian Cartan geometries}
Carrollian geometries, as obtained on $\Sigma^{-}$ and $\Sigma^+$ by Theorem \ref{Theorem: induced geometry on Sigma}, define unique normal Cartan connections respectively modelled on the homogeneous spaces $\mathrm{Ti}^n$ and $\mathrm{Spi}^n$ defined by Eq~\eqref{Boundary Homogeneous space}).
\end{prop}
Without much surprise these Cartan geometries coincide with the induced tractor geometries given by Proposition \ref{Proposition: Induced tractor geometry on Sigma}.
\begin{prop}
	The tractor connections respectively induced on $\T_{\Sigma^-} \to \Sigma^-$ and $\T_{\Sigma^+} \to \Sigma^+$ coincides with the normal Cartan connections associated with the geometry of Theorem \ref{Theorem: induced geometry on Sigma}.
\end{prop}
\begin{proof} In the proof of Theorem \ref{Theorem: induced geometry on Sigma} we already saw that, in the scale given by $\lambda$, the induced tractor pairing $\Phi_{AB}$ is in the diagonal form \eqref{Proof: tractor metric on Sigma} while $I^A$ is simply proportional to $n^a$. Moreover, the linear connection $\nabla_a$ induced by the tractor connection is metric compatible. It follows that the induced tractor connections define Cartan geometries modelled on $\mathcal{M}_{\pm1,0}$.
	
	 We thus only need to prove that the induced tractor connection is normal, which in this case just reduces to torsion-freeness. However it is straightforward to see that the induced tractor connection on $\Sigma$ is torsion-free as a result of the fact that the tractor connection on $M$ is.
\end{proof}

\subsection{The geometry of $\mathcal{H}$}

By definition $\mathcal{H}^{\pm}$ are obtained by taking the quotient of $\Sigma^{\pm}$ by the flow of $n^a$. What is more, by Corollary \ref{Corollary: Densities induced on Q} the boundary defining function $\sigma$ of $\Sigma$  descends to a boundary defining $\bar{\sigma}$ function for $\mathcal{H}$. 

By taking the corresponding quotient by $n^a$ and $I^A$, the short sequences of Proposition \ref{Proposition: short sequences of tractors on Sigma} then descends to
\begin{prop}\mbox{}
\begin{center}\label{Diagramm: tractor bundles along H}
	\begin{tikzcd}
	&0\arrow[d] & 0\arrow[d]\\&\mathcal{E}_{\mathcal{H}}(-1)\arrow[r,"\simeq"] \arrow[d,hook,"X"]& i^*\E(-1)\equiv\E(-1)|_{\mathcal{H}} \arrow[d,hook,"X"] \\ 0 \arrow[r]& \T_{\mathcal{H}}\simeq i^*(\tilde{I}_A)^{\circ} \arrow[r,hook] \arrow[d,two heads, "Z"] & i^*\T_{\Q} \simeq i^*(\tilde{\T}/\tilde{I}) \arrow[r,two heads,"\bar{D}_A\bar{\sigma}"] \arrow[d,two heads, "Z"]& C^\infty(\mathcal{H})\arrow[d,"\simeq"] \arrow[r]& 0
	\\0\arrow[r]& T\mathcal{H}(-1)\arrow[d]\arrow[r,hook]& i^*T\Q(-1)\arrow[d] \arrow[r,"\bar{\nabla}_a\bar{\sigma}", two heads]& C^\infty(\mathcal{H})\arrow[r] & 0\\&0&0
	\end{tikzcd}
\end{center}
\end{prop}
\begin{proof}
	This follows from Proposition \ref{Proposition: tractor isomorphism}, Proposition \ref{Proposition: short sequences of tractors on Sigma} and the fact that $\tilde{I}_A = D_A \bar{\sigma}$.
\end{proof}

The relationship between the different tractor bundles that we discussed up to now are summarised by the following commutative diagram
\begin{figure}[h!]
	\centering
	\begin{tikzcd}
i^*(\mathcal{T}_{M}) \arrow[d, two heads] & \arrow[l, hook] \mathcal{T}_{\Sigma} \simeq i^*( I_A)^{\circ}  \arrow[d, two heads]\\ 
i^*\Big(\pi^*\T_{\Q}\Big) \simeq i^*(\mathcal{T}_{M}/I)  & \arrow[l, hook] \pi^*\T_{\mathcal{H}} \simeq i^*( I_A)^{\circ}/I.
\end{tikzcd}
	\caption{\label{Figure: The tractor bundles along the boundary} The tractor bundles along the boundary $\Sigma \to \mathcal{H}$.}
\end{figure}

From Proposition \ref{Proposition: connection on Q} the tractor connection on $\Sigma$ induces a tractor connection on $\mathcal{H}$. Now, by Theorem~\ref{Theorem: Projective compactification order 1}, $\mathcal{H}$ is the projective boundary of a projectively compact Einstein metric which, by \cite{Cap:2014ab}, also induces a tractor connection on $\mathcal{H}$. From the above relationship between the tractor bundle and our previous results these connections coincide and, by the results of \cite{Cap:2014ab} one has:

\begin{coro}\label{Proposition: Cartan geometries on H}
	The projective tractor connections induced on $\T_{\mathcal{H}^-} \to \mathcal{H}^-$ and $\T_{\mathcal{H}^+} \to \mathcal{H}^+$ inherit a holonomy reduction to $SO(n-2,1)$  and respectively define normal Cartan connections modelled on $H^{n-1}$ and $dS^{n-1}$. From \cite{Cap:2014ab} these correspond to projectively compact Einstein manifolds with projective boundary $\mathcal{H}^0$.
\end{coro}

\subsection{The geometry of $\mathcal{I}\to \mathcal{H}$}

According to Corollary \ref{Corollary: Densities induced on Q}, $\tilde{\sigma}\in \Gamma(\mathcal{E}_{\Q}(1))$ is a boundary defining function for $\mathcal{H} = \mathcal{H}^{-}\cup\mathcal{H}^{0}\cup \mathcal{H}^{+}$. This implies the short sequence (dual to the second line of \eqref{Diagramm: tractor bundles along H}):
\begin{center}
	\begin{tikzcd}
C^{\infty}(\mathcal{H}) \arrow[r,hook,"J^1 \bar{\sigma}"]& i^*(\T_{\Q})^* \simeq i^*(J^1 \mathcal{E}_{\Q}(1)) \arrow[r,two heads,"\Pi"]  & \frac{i^*(J^1 \mathcal{E}_{\Q}(1))}{ i^* (J^1\bar{\sigma})} \simeq J^1 \mathcal{E}_{\mathcal{H}}(1) \simeq (\T_{\mathcal{H}})^*.
	\end{tikzcd}
\end{center}

 What is more, from Corollary \ref{Corollary: Densities induced on Q},  $\sqrt{\bar{\lambda}}\in \Gamma(\mathcal{E}_{\mathcal{H^{\pm}}}(1))$ is a nowhere vanishing smooth density on $\mathcal{H}^{\pm}$. This defines a section
 
 \begin{equation}\label{Jlambda section}
 	\phi:  \mathcal{H}^{\pm} \xrightarrow{J^1 \sqrt{\bar{\lambda}}} J^1 \mathcal{E}_{\mathcal{H}}(1) \simeq \frac{i^*(J^1 \mathcal{E}_{\Q}(1))}{ i^* (J^1\bar{\sigma})}.
 \end{equation}
 
 However, there is no canonical way to extend $\sqrt{\bar{\lambda}}$ as a density in a neighbourhood of $\mathcal{H}$ in $\Q$ and therefore no preferred way to extend $J^1 \sqrt{\bar{\lambda}}$ to a section of $i^*(J^1 \mathcal{E}_{\Q}(1))$, that we will view as a line bundle over $J^1\E_{\mathcal{H}}(1)$. This fact is the basis of the following definition.
\begin{defi}
The line bundle $\mathcal{I}^{\pm} \to \mathcal{H}^{\pm}$ is defined as the pull-back bundle
\begin{equation}
\phi^* \Big( i^*(J^1 \mathcal{E}_{Q}(1)) \Big) \to \mathcal{H}^{\pm}.
\end{equation}
\end{defi}
Sections of $\mathcal{I}^{\pm}$ therefore correspond to choosing an extension of $J^1 \sqrt{\bar{\lambda}}$ as a section of $i^*(J^1 \mathcal{E}_{\Q}(1))$. This is naturally a $\mathbb{R}$-principal bundle: if one notes $J^1\sqrt{\Lambda}$ the 1-jet extension of $\sqrt{\bar{\lambda}}$ in $i^*(J^1 \mathcal{E}_{\Q}(1))$ corresponding to a point in $\mathcal{I}^{\pm}$, then the $\mathbb{R}$-action is given by
\begin{equation}\label{R action on I}
\begin{array}{ccc}
\mathbb{R}\,\times\, \mathcal{I}^{\pm}  & \to & \mathcal{I}^{\pm}\\[0.4em]
(k\,,\, J^1\sqrt{\Lambda}) & \mapsto &  J^1\sqrt{\Lambda} + k \,J^1\sigma.
\end{array}
\end{equation}

This bundle is intrinsic to $\mathcal{H}^{\pm}$ as a submanifold of $\Q$ in the sense that it only depends on data coming from the projective compactification. Nevertheless it also nicely ties up with the geometry of $M$.
\begin{prop}\label{Proposition: scrI.Sigma equivalence}
$\mathcal{I}^{\pm}$ is canonically isomorphic to $\Sigma^{\pm}$. 

Furthermore, the action \eqref{R action on I} coincides, through this isomorphism, to the $\mathbb{R}$-action induced on $\Sigma^{\pm}$ by the flow of $\sqrt{\lambda}n^a$.
\end{prop}

\begin{proof}
In the previous sections we found that points in the fibre $i^*(J^1\mathcal{E}_{\Q}(1))_x$ at $x\in \mathcal{H}^{\pm}$ correspond to covariantly constant sections of $(I^A)^{\circ}_{\pi^{-1}(x)} \subset (\T_{M})^*_{\pi^{-1}(x)}$ while points in the fibre $J^1\mathcal{E}_{\mathcal{H}}(1)_x$ correspond to covariantly constant sections of $\big((I^A)^{\circ}/I_A\big)_{\pi^{-1}(x)}$. In particular, since we have
\begin{equation}
\nabla_{n} D_{A}\sqrt{\lambda} = I^B D_{B}\Big( \tfrac{1}{2} \lambda^{-\frac{1}{2}}D_{A}\lambda\Big)=  -\tfrac{\lambda^{-\frac{3}{2}}}{4}I^B D_{B}\lambda \;D_{A}\lambda+\lambda^{-\frac{1}{2}}I^B D_{B}D_{A}\lambda
\end{equation}
which implies (by making use of Lemma \ref{Lemma: Relations ISigma, XLambda, etc})
\begin{equation}
\Big(\nabla_{n} D_{A}\sqrt{\lambda}\Big)\Big|_{\Sigma} =  0+\lambda^{-\frac{1}{2}}I_A
\end{equation}
it follows that $D_{A}\sqrt{\lambda}\big|_{\Sigma}$ defines a covariantly constant section of $i^*\big((I^{A})^{\circ}/I_A\big) \simeq (\T_{\mathcal{H}})^*$. One can easily prove that this section is identified with the section \eqref{Jlambda section}.

From this remark and the definition of $\mathcal{I}^{\pm}$ we thus find that sections of $\mathcal{I}^{\pm}$ are identified with sections of $i^*(I^{A})^{\circ} \subset i^*(\T_{M})^*$ of the form
 \begin{equation}\label{proof: scrI.Sigma equivalence chi def}
 D_{A}\sqrt{\lambda} + \chi I_A,
 \end{equation} where $\chi$ is a function on $\Sigma$, and which are covariantly constant along the fibres of $\Sigma^{\pm} \to \mathcal{H}^{\pm}$. This last condition is equivalent to requiring that $\chi$ satisfies
\begin{equation}\label{proof: scrI.Sigma equivalence edp}
\nabla_N \chi = -1
\end{equation}
where $N^a = \lambda^{\frac{1}{2}}n^a$. In a coordinate system $(s,y)$ adapted to $\Sigma^{\pm} \to \mathcal{H}^{\pm}$, and such that $N^a= \partial_s$, solutions are of the form
\begin{equation}\label{proof: scrI.Sigma coordinate expression for chi}
\chi = \xi(y) - s.
\end{equation}
It follows that the set of points where $\chi$ vanishes defines a section $y \mapsto (\xi(y), y)$ of $\Sigma^{\pm} \to \mathcal{H}^{\pm}$. Such sections are thus in one-to-one correspondence with solutions of \eqref{proof: scrI.Sigma equivalence edp} and therefore in one-to-one correspondence with sections of $\mathcal{I}^{\pm}$. 

Finally comparing \eqref{R action on I} with \eqref{proof: scrI.Sigma equivalence chi def} one sees that the $\mathbb{R}$-action $\eqref{R action on I}$ amounts to adding $k\in \mathbb{R}$ to $\chi$. The coordinate expression \eqref{proof: scrI.Sigma coordinate expression for chi} then shows that this is equivalent to the action of the flow of $N^a= \partial_s$.
\end{proof}

\begin{defi}
	The $\mathbb{R}$-action allows to parallel transport 1-jets $J^1\mathcal{E}_{\mathcal{I}^{\pm}}(1)$ on $\mathcal{I}^{\pm}$, we note \begin{equation*}
	(\bm{T}_{\mathcal{H}})^* \to \mathcal{H}
	\end{equation*} the resulting bundle on $\mathcal{H}$. 
\end{defi}

\begin{prop}\label{Proposition: Cartan geometries on I}	
$\bm{T}_{\mathcal{H}} \to \mathcal{H}^{\mp}$ coincides with the bundle $\tilde{\T}_{\mathcal{H}} \to \mathcal{H}^{\mp}$ obtained by taking the quotient of $\T_{\Sigma} \to \Sigma^{\mp}$ by the flow of $n^a$ (as in section \eqref{sssection: Tractors of Q}). In particular it is equipped with a normal tractor connection modelled on the homogeneous space $\mathrm{Ti}^n$, $\mathrm{Spi}^n$  (see \eqref{Boundary Homogeneous space}). 
\end{prop}
\begin{proof}
It follows from the definition and Proposition \ref{Proposition: scrI.Sigma equivalence}.
\end{proof}
From the diagram in Figure \ref{Figure: The tractor bundles along the boundary} we then find:
\begin{lemm}
	One has a canonical isomorphism
	\begin{equation*}
	\bm{T}_{\mathcal{H}} \simeq i^*(\tilde{I}_A)^{\circ} \subset i^* \tilde{\mathcal{T}}.
	\end{equation*}
\end{lemm}
\noindent Since, on $\Q_{\mathcal{O}}$ and by Proposition \ref{Proposition: ambient tractors extend conformal},  $\tilde{\mathcal{T}}$ is isomorphic to the conformal tractor bundle,  the preceding Lemma means that $\bm{T}_{\mathcal{H}}$ is in a certain sense the extension by continuity of this conformal tractor bundle along the projective boundary. Similarly the Cartan connection induced on these bundles by Proposition \ref{Proposition: Cartan geometries on I}	must be an extension by continuity of the conformal tractor connection.

The importance of these last results stems from the fact that normal tractor connections modelled on the homogeneous space $\mathrm{Ti}^n$ and $\mathrm{Spi}^n$ are not unique, see e.g. \cite{Herfray:2021qmp}, not even after fixing a flat normal tractor connection on $\mathcal{H}^{\pm}$. The resulting freedom encodes asymptotic data at space/time-like infinity induced by the projective compactification.

\section{Asymptotic symmetries of projectively compact Einstein manifolds of order 1}\label{section: Asymptotic symmetries of projectively compact Einstein manifolds}

Let $(\Q, \bar{\nabla}, \bar{H}^{AB}, \bar{I}_{A} )$ be a projectively compact order 1 Einstein manifold (i.e. as obtained in the previous sections\footnote{We have not shown in this article that all such manifolds arise in this way. There is, however, a sense in which this is true; it turns out that all such manifolds admit a canonical non-effective Cartan geometry of the type conjectured in Section~\ref{AdditionalTractorBundle}. This result will appear in~\cite{BorthwickHerfray2}. }).  We will denote by $Sym(\Q)$ the space of diffeomorphisms $\Phi : \Q \to \Q$ preserving the boundary $\Phi(\mathcal{H}^{\pm}) = \mathcal{H}^{\pm}$ and such that
\begin{align}
\left(\Phi^* \bar{H}^{AB}\right) \big|_{\mathcal{H}^{\pm}} & = \bar{H}^{AB}\big|_{\mathcal{H}^{\pm}}, & \left(\Phi^* \bar{I}_{A}\right) \big|_{\mathcal{H}^{\pm}} & = \bar{I}_{A}\big|_{\mathcal{H}^{\pm}}.
\end{align}
In particular, these diffeomorphisms must induce an isometry $\phi : \mathcal{H}^{\pm} \to  \mathcal{H}^{\pm}$ of $(\mathcal{H}^{\pm} , \lambda \zeta^{ab})$ and therefore preserve the induced Cartan geometry given by Proposition \ref{Proposition: Cartan geometries on H} (but will not in general preserve the Cartan geometry induced on $\mathcal{I}^{\pm}$ by Proposition \ref{Proposition: Cartan geometries on I}).
We will note $Sym_0(\Q)$ the space of diffeomorphisms $\Phi : \Q \to \Q$ which act trivially on the boundary up to first order $i^*(d\Phi) = Id$. 

\begin{defi}
We define the set of asymptotic symmetries of a projectively compact Einstein manifold $(\Q, \bar{\nabla}, \bar{H}^{AB}, \bar{I}_{A})$ to be the quotient $Asym(\Q) = Sym(\Q)/Sym_0(\Q)$.
\end{defi}

\begin{theo}
The action of the asymptotic symmetries on $\mathcal{I}^{\pm}$ coincides with the action of the principal bundle automorphisms of $\mathcal{I}^{\pm}$ preserving the Einstein metric on $\mathcal{H}^{\pm}$ and we thus have
\begin{align*}
Sym(Q)/Sym_0(Q) &\simeq C^{\infty}(\mathcal{H}^{\pm}) \rtimes ISO(\lambda\zeta^{ab})
\end{align*}

In particular, the group of asymptotic symmetries naturally acts on the normal Cartan connections induced on $\mathcal{I}^{\pm}$ by Proposition \ref{Proposition: Cartan geometries on I}.
\end{theo}
\begin{proof}
	
Let $\Phi:\Q \to \Q$ be an asymptotic symmetry, we note $\phi: \mathcal{H} \to \mathcal{H}$ the induced action on $\mathcal{H}$. Since $\Phi$ preserves $i^*(\bar{H}^{AB})$ it follows that $\phi$ must be an isometry of $\lambda \zeta^{ab}$.

We recall that along the boundary $\mathcal{H}$ the restriction of the tractor metric $\bar{H}^{AB}$ to $\mathcal{T}_\mathcal{H}$ is invertible of inverse $\tfrac{1}{2}D_AD_B\bar{\lambda}$ and that $D_A\bar{\lambda} = (D_AD_B\bar{\lambda}) X^B$. It follows that an asymptotic symmetry must preserve both $i^*I_A$ and $D_A\sqrt{\bar{\lambda}}$. 

Let $J^1_x\sqrt{\Lambda}$ be a point in $\mathcal{I}^{\pm}$ i.e. a point in $J^1_x \mathcal{E}_{\Q}(1)$ whose quotient by $J^1_x\bar{\sigma}$ coincides with  $J^1_x\sqrt{\lambda}$. 
By the previous remark, the image of this point by an asymptotic symmetry must of the form
\begin{equation}
\Phi^* (J^1\sqrt{\Lambda})_x = J^1_{x}\sqrt{\Lambda} + \chi J^1_{x} \sigma
\end{equation} 
where $\chi$ is some function that might at this stage depend on the points of the fibres of $J^1_x\sqrt{\Lambda} \to J^1_x\sqrt{\lambda}$. To conclude we need to prove that symmetry is compatible with the $\mathbb{R}$ action $J^1\sqrt{\Lambda} \mapsto J^1_x\sqrt{\Lambda} + k J^1_{x} \sigma$, $k\in \mathbb{R}$. However,
\begin{equation}
\Phi^* (J^1_x\sqrt{\Lambda} + k J^1_{x} \sigma ) = J^1_{\phi(x)}\sqrt{\Lambda} + (\chi + k)  J^1_{\phi(x)} \sigma
\end{equation} 
follows from the fact that a symmetry must preserve $i^*(I_A)$.  Therefore, the asymptotic symmetry induces an automorphism of the bundle. One concludes the proof by observing that such automorphism completely fixes the first order jets of the symmetry and is thus equivalent to an asymptotic symmetry.
\end{proof}

\printbibliography[title=Bibliography]

 \end{document}